\documentclass[11pt]{article}
\usepackage{amsmath, amssymb, amsthm}
\usepackage{enumitem}

\usepackage[margin=3.3cm]{geometry}

\usepackage{cite}

\newcommand{\RR}{\mathbb{R}}
\newcommand{\NN}{\mathbb{N}}
\newcommand{\ZZ}{\mathbb{Z}}
\newcommand{\QQ}{\mathbb{Q}}

\newtheorem{theorem}{Theorem}[section]

\newtheorem{Counter-example}[theorem]{Counter example}

\newtheorem{Lemma}[theorem]{Lemma}
\newtheorem{Proposition}[theorem]{Proposition}

\newtheorem{Corollary}[theorem]{Corollary}

\newtheorem*{theorem*}{Theorem}

\DeclareMathOperator*{\cmpct}{cpct}
\DeclareMathOperator*{\diag}{diag}

\newcommand\blfootnote[1]{%
  \begingroup
  \renewcommand\thefootnote{}\footnote{#1}%
  \addtocounter{footnote}{-1}%
  \endgroup
}

\title{Self embeddings of Bedford-McMullen carpets}

\author{Amir Algom and Michael Hochman}

\begin{document}
\maketitle

\begin{abstract}
Let\blfootnote{Supported by ERC grant 306494} $F \subseteq \mathbb{R}^2$ be a Bedford-McMullen carpet defined by multiplicatively independent exponents, and suppose that either $F$ is not a product set, or it is a product set with  marginals  of dimension strictly between $0$ and $1$.  We prove that any similarity $g$ such that $g(F) \subseteq F$ is an isometry composed of reflections about lines parallel to the axes. Our approach utilizes the structure of tangent sets of $F$, obtained by "zooming in" on points of $F$,  projection theorems for products of self-similar sets, and logarithmic commensurability type results for self similar sets in the line.

\end{abstract}

\tableofcontents

\section{Introduction}

Let $F \subseteq \mathbb{R}^2$ be a dynamically defined fractal, such as a smooth repeller or  self affine set. The problem we address in this paper is to classify those maps $g:F \rightarrow F$ satisfying $g(F) \subseteq F$. In many natural cases the hierarchical structure of $F$ and its defining parameters impose severe restrictions on $g$, and one may expect that $g$ should in some sense ``come from'' the generating dynamics, either belonging to the acting (semi-)group or to some larger class of maps related to it.

An early and motivating example of such a situation is Furstenberg's theorem  \cite{furstenberg1967disjointness} that if a closed subset $X\subseteq \mathbb{R}/\mathbb{Z}$ is non-trivial (neither finite or all of $\mathbb{R}/\mathbb{Z}$), and is invariant under an endomorphism $T$ of $\mathbb{R}/\mathbb{Z}$, then another endomorphism $S$ can map $X$ into itself if and only if $S,T$ lie in a common cyclic semi-group. Many generalizations of this exist to other algebraic contexts. In a slightly different direction,  the second author showed in \cite{hochman2010geometric} that if $\mu$ is an invariant measure of positive entropy for such an endomorphism $T$, and $g$ is a piecewise-monotone  $C^2$-map such that $g\mu\ll\mu$, then $\mu$-a.e. the derivatives of $g$ and $T$ agree.

More closely related to the present paper is the work of Feng and Wang \cite{feng2009structures}, who considered self-similar sets on the line and asked to what extent they determine the iterated function systems (IFSs) that generate them. An important component of their work, which will play a role in ours, is the so-called Logarithmic Commensurability Theorem (stated precisely in Theorem \ref{LCT} below). It asserts that if a self-similar set in the line is generated by maps which contract by $t$ (and satisfy certain mild conditions), then any similarity mapping the set into itself must contract by a rational power, and in some cases integer power, of $t$. Closely related to this is a conjecture of a Feng, Huang, and Rao \cite{feng2014affine}, that if two self similar sets in the line are defined by maps which contract, respectively, by $t_1$ and $t_2$, then (again under some assumptions) any similarity map embedding one into the other must contract by a rational power of both $t_1$ and $t_2$, and in particular, $\log t_1/\log t_2$ must be rational.

Another relevant paper is the that of  Elekes, Keleti and M\'{a}th\'{e} in \cite{elekes2010self}. They studied self-similar sets in $\mathbb{R}^d$, and, assuming strong separation, showed that if a similarity maps such a set into itself, then the image has non-empty interior, and some power of the map lies in the group generated by the maps in the IFS that originally  defined the set.

In this paper we consider the analogous question for Bedford-McMullen carpets $F$, under the assumption that the bases $m,n$ used to define it are multiplicatively independent. This means that $\frac{\log m}{\log n}\notin\mathbb{Q}$, and that $F$ is obtained by an iterative procedure, first partitioning the unit square  into an $m\times n$ grid of sub-rectangles and discarding a subset of them, repeating the procedure for each remaining sub-rectangle using the same pattern of discarded sub-rectangles as in the first stage, and continuing \emph{ad infinitum}. For a formal definition and other descriptions of $F$, see Section \ref{Bedford - McMullen carpets}. Our main result is the following:

\begin{theorem} \label{Main Theorem}
Let $m,n$ be multiplicatively independent integers and let $F$ be a Bedford-McMullen carpet which is not a product set. Then any similarity $g$ taking $F$ into itself is an isometry, and must be either the identity, a reflection through a horizontal or vertical line, or a composition of two such reflections.
\end{theorem}

The assumption that $F$ is not a product set excludes the cases when $F$ is a point or all of $[0,1]^2$, for which the conclusion is obviously false. It also excludes the case that $F$ is supported on a horizontal or vertical line, in which case one can show that it is a self-similar set on that line, and admits many non-trivial similarities preserving it. Also, without the assumption that the defining bases are multiplicatively independent, $F$ may be self-similar, and again will have many non-trivial similarities taking it into itself. We also note that $F$ can indeed be symmetric with respect to reflection about its vertical and horizontal bisecting lines, as occurs when the set of sub-rectangles removed at each step of the construction has the same kind of symmetry.

Nevertheless, we can say the same thing for product sets if we avoid trivial marginals. The statement we prove in this case is slightly stronger, in that it restricts self-embeddings by affine maps, not just than similarities:

\begin{theorem} \label{Theorem product sets}
Let $m,n$ be multiplicatively independent integers and $F$ a Bedford-McMullen carpet which is a product $F=K_1\times K_2$. If $0<\dim K_1,\dim K_2<1$, then any invertible affine map $g$ such that $g(F)\subseteq F$ has linear part $\diag(\pm m^p, \pm n^q)$ for $p,q \in \mathbb{Q}$. In particular, if $g$ is a similarity map then $p=q=0$ and $g$ is an isometry as in Theorem \ref{Main Theorem}.
\end{theorem}

The assumptions are again necessary to avoid trivial exceptions, such as self-similar sets supported on lines, for which the defining similarities extend in many ways to affine maps off of the supporting line, or products of an interval with a self-similar set, where the same thing happens. Independence of the bases is also necessary; for instance, consider the standard middle-$1/3$ Cantor set $C$; then $C\times C$ is a Bedford-McMullen carpet, yet there are many similarities taking it into itself, e.g. any map of the form $(x,y)\mapsto (x/3^k,y/3^k)$.

The mechanism that we use is essentially different from that used by Elekes et al in the self-similar case. There, the crucial ingredient was that when $F$ is a self-similar set with strong separation and dimension $s=\dim F$, the restriction $\mu=\mathcal{H}^s|_F$ of the $s$-dimensional Hausdorff measure is positive and finite. If $g$ is a similitude then $g(F)$ also has positive finite $\mu$-measure, and we can conclude from $g(F) \subseteq F$ that $g\mu\ll\mu$. One then can apply analytic techniques to study ``density points'' of $d(g\mu)/d\mu$, combined with re-scaling (using self-similarity of $F$) to gain control of the map. In fact, Elekes, Keleti and M\'{a}th\'{e} obtained analogous results for Bedford-McMullen carpets $F$ and similitudes $g$ that map some self-affine measure $\mu$ to a measure $g\mu\ll\mu$. However, for self-affine measures the condition $g(F) \subseteq F$ no longer implies that $g(F)$ has positive $\mu$-measure, nor would this imply the required absolute  continuity; and the more natural Hausdorff measure is infinite \cite{Peres1993}. So this method cannot be used to prove Theorem \ref{Main Theorem}.

Instead we pursue a more geometric analysis. An important input to our proof is a description of the tangent sets of $F$ (Hausdorff limits of suitably re-scaled balls in $F$), and of the horizontal and vertical slices of $F$. By a careful analysis of how these sets are transformed by $g$ and how they can be mapped into each other, we are able first to show that the linear part of $g$ is a diagonal matrix (so $g^2$ is a homothety), and then that the contraction must be $\pm 1$. The main ingredient in the first stage is the projection theorems of Peres-Shmerkin \cite{peres2009resonance}, Hochman-Shmerkin \cite{hochman2009local}, applied to the tangent sets. For the second part we apply one dimensional results, like the Logarithmic Commensurability Theorem mentioned earlier, to slices and projections of $F$. We remark that the tangent sets of Bedford-McMullen carpets have been described in several recent works, e.g. Bandt and K\"{a}enm\"{a}ki  \cite{bandt2013local} and K\"{a}enm\"{a}ki, Koivusalo and Rossi \cite{kaenmaki2015self}, but there the authors considered the limits taken by magnifying the set around  typical point for a self-affine measure. For our application we require a description of the limit sets at every point.

After this work was completed we became aware of a parallel project by K{\"a}enm{\"a}ki, Ojala and Rossi \cite{kaenmaki2016rigid}. They use similar techniques to show that, in a certain class of self-affine sets $F$ (similar but not comparable to the one we study), any quasi-symmetric map taking $F$ into $F$ is quasi-Lipschitz, which means that $d(g(x),g(y))= d(x,y)^{1+o(1)}$ as $x\to y$. This quasi-Lipschitz property is weaker than being Lipschitz and certainly doesn't imply it, so comparing our results to theirs, we see that we obtain stronger information (the maps are isometries) starting from a more restricted class of self-embeddings (similarities). Nevertheless, many of the methods are similar, in particular the use of tangent sets as an invariant of regular maps.

\textbf{Organization} The next section contains the basic definitions and background material. Section \ref{Section - embedding deleted digit set} contains an auxiliary  theorem about the translation part of affine self-embeddings of certain self-similar sets in the line. Section \ref{subsec:tangent-sets} introduces the definition and basic properties of tangent sets of Bedford-McMullen carpets. Theorem \ref{Theorem product sets} is proved in Section \ref{Section - proof of product set theorem}, followed by Theorem \ref{Main Theorem} in Section \ref{Section - proof of main Theorem}, and  Theorem \ref{Theorem - structure of tangnet sets}, about the structure of tangent sets, is proved in Section \ref{Section - proof of Theorem tangent sets}.

\textbf{Acknowledgment} This research was conducted as part of the first author's Ph.D. studies. The first author would like to thank Shai Evra for many helpful discussions. The authors are grateful for the hospitality and support received from ICERM as part of the spring 2016 program on dimension and dynamics.

\section{Preliminaries and notation}
\subsection{Iterated function systems} \label{digit restriction definition}
Let $\Phi = \lbrace \phi_i \rbrace_{k=1} ^l, l\in \mathbb{N}, l\geq 2$ be a family of contractions $\phi_i : \mathbb{R}^d \rightarrow \mathbb{R}^d, d\geq 1$. The family $\Phi$ is called an iterated function systems, abbreviated IFS. There exists a unique compact set $\emptyset \neq F \subseteq \mathbb{R}^d$ such that $F = \bigcup_{i=1} ^l \phi_i (F)$, called the attractor of $\Phi$, and $\Phi$ is called a generating IFS for $F$.  A cylinder set is defined to be a set of the form $\phi_{i_1} \circ ... \circ \phi_{i_k} (F)$, where $\phi_i \in \Phi$ for all $i$ and $k\in \mathbb{N}$. Writing $I= (i_1,...,i_k)\in [l]^k$ and denoting $\phi_I = \phi_{i_1} \circ ... \circ \phi_{i_k}$, cylinder sets have the form $\phi_I (F), I\in [l]^*$.  Every open non empty set in $F$ contains a cylinder set. 

A map $g :\mathbb{R}^2 \rightarrow \mathbb{R}^2$ is a similarity map if 
\begin{equation*}
g(x) = \alpha\cdot O(x) + t
\end{equation*}
where $\alpha > 0$, $t \in \mathbb{R}^2$  and $O$ is an orthogonal matrix. We call $\alpha$ the contraction, $O$ its linear part, and $t$ the translation, of $g$. More generally if $g$ is an affine map, i.e.
\begin{equation*}
g(x)= A(x)+t
\end{equation*}
with $A\in GL(\mathbb{R}^2)$ and $t\in \mathbb{R}^2$, then $A$ is called the linear part of $g$ and $t$ its translation part. In this paper all affine maps are non-singular (though we may state this explicitly for emphasis).

A set $F \subseteq \mathbb{R}^d$ will be called self similar if there exists a generating IFS $\Phi$ for $F$ such that every $\phi \in \Phi$ is a similarity map. Similarly, a set $F \subseteq \mathbb{R}^d$ will be called self affine if there exists a generating IFS $\Phi$ for $F$ such that every $\phi \in \Phi$ is an affine map. Thus, self similar sets are self affine, but the converse is not true in general. 

\subsection{Deleted-digit sets}
Let $m \in \mathbb{N}$ be such that $m\geq 2$, and denote 
\begin{equation*}
[m]=\lbrace 0,...,m-1 \rbrace 
\end{equation*}
Let $\Lambda \subseteq [m]$. The deleted digit set with base $m$ and digits $\Lambda$ is the set $D(\Lambda ,m)  \subseteq [0,1]$ of real numbers in $[0,1]$ that  admit an expansion in base $m$ that uses only the digits in $\Lambda$. Explicitly,
\begin{equation} \label{Set defined by digit restriction}
D(\Lambda,m) = \lbrace \sum_{k=1} ^\infty \frac{x_k}{m^k} \; |\;  x_k \in \Lambda \rbrace.
\end{equation}
For example, the middle thirds Cantor set, for example, is equal to $D (\lbrace 0,2 \rbrace,3)$. It is not hard to verify that $D(\Lambda ,m)$ is the attractor of the IFS $\Phi = \lbrace \phi_j \rbrace_{j\in \Lambda}$ where $\phi_j (x) = \frac{x+j}{m}$.
Thus, $D(\Lambda ,m) $ is a self similar set, and the IFS above satisfies the open set condition (using the open set $U = (0,1)$) and so by e.g. Hutchison \cite{hutchinson1979fractals}, it follows that 
\begin{equation*}
\dim D(\Lambda,m) = \frac{\log |\Lambda|}{\log m}
\end{equation*}
where by $|\Lambda|$ we mean the cardinality of the finite set $\Lambda$, and by $\dim$ we mean the Hausdorff dimension. 

Let $\Omega_m = {[m]}^\mathbb{N}$ denote the space of one-sided sequences in the alphabet $\lbrace 0,...m-1 \rbrace$, which is a compact topological space in the product topology. Let $\sigma_m : \Omega_m \rightarrow \Omega_m$ be the left shift,  defined for $\omega\in\Omega_m$ by
\begin{equation*}
(\sigma_m ( \omega ))_{p} = \omega_{p+1}.
\end{equation*}
Define the ``projection'' $\pi_m : \Omega_m \rightarrow [0,1]$ by
\begin{equation*}
\pi_m (\omega) = \sum_{k=1} ^\infty \frac{\omega_k}{m^k}, \quad \omega= (\omega_1 ,...).
\end{equation*}
If $K = D(\Lambda,m)$ then it is the image of  $\widetilde{K} = \Lambda^\mathbb{N} \subseteq \Omega_m$ under $\pi_n$, i.e. $K=\pi_n(\widetilde{K})$.
The map $\pi_m |_{\widetilde{K}}$ is a continuous surjection to $K$, but can fail to be injective on countably many points, specifically, rationals in $(0,1)$ of the form $k/m^n$ have two preimages under $\pi_m$ (but note that $0,1$ have only one pre-image). Also, we remark that $\widetilde{K} \subseteq \pi_m ^{-1} (K)$, but the two sets might not be equal.

We shall call rationals in $(0,1)$ of the form $k/m^n$, where $0\leq k\leq m^n$ is an integer, $m$-adic rationals.  We extend the definition of  $\pi_m$ to finite sequences, so if $a\in[m]^k$ then $\pi_m(a)=\sum_{i=1}^ka_im^{-i}$. Thus the $m$-adic rationals are precisely the images $\pi_m(a)$ of $a\in[m]^*$.

We shall say that $x\in[0,1]$ has a unique expansion in base $m$ if it has a unique pre-image under $\pi_m$. This differs slightly from the usual usage at the point $0$ and $1$: by our definition they have a unique expansion, whereas usually they are considered to have two. The difference is that we only consider expansions ``without an integer part''.

\subsection{Logarithmic Commensurability}

Several results are known which give algebraic constraints on maps taking self-similar sets on the line into themselves or into each other. Here we state a simplified version for deleted digit sets, which says that such maps are only possible when the bases are compatible. 

\begin{theorem}[Feng-Wang \cite{feng2009structures}] \label{LCT}
Let $K=D(\Lambda,m)$ with $0<\dim K<1$ and let $g(x)=\alpha x+t$ be a similarity with $gK\subseteq K$. Then $\frac{\log |\alpha|}{\log m}\in \mathbb{Q}$.
\end{theorem}

The following Theorem states that deleted digit sets defined by multiplicatively independent bases do not admit affine embeddings into each other (apart from the trivial ones):
\begin{theorem}[Feng-Huang-Rao \cite{feng2014affine}, Hochman-Shmerkin \cite{hochmanshmerkin2015}]\label{LCT2}
  Let $K_1=D(\Lambda,m)$ and $K_2=D(\Delta,n)$. If $0< \dim K_1 ,\dim K_2<1$ and $\log m/\log n\not\in\mathbb{Q}$, then any similarity mapping $K_1$ into $K_2$ maps to a single point.
\end{theorem}

Theorem \ref{LCT} has the following easy extension:

\begin{Corollary} \label{Corollary from LCT}
Let $K=D(\Lambda,m)$ and $0<\dim K<1$. Let $g(x) = \alpha x +t$ be an affine map of the line such that for some $c_i\in\RR$,
\begin{equation*}
K \subseteq \bigcup_{i=1}^\infty (g(K) +c_i).
\end{equation*}
Then $\frac{\log |\alpha|}{\log m} \in \mathbb{Q}$.
\end{Corollary}
\begin{proof} 
By Baire's Theorem  there exists some $c\in \mathbb{R}$ such that $g(K)+c$ has non empty interior in $K$. Therefore, it contains a cylinder set $\phi_{i_1} \circ ... \circ \phi_{i_k} (K) \subseteq g(K)+c$. Denote $\phi_{i_1} \circ ... \circ \phi_{i_k} (x) = s (x) = \frac{1}{m^k} \cdot x + r, r\in \mathbb{R}$. Since $s (K) \subseteq g(K)$ we have $ g^{-1} \circ s (K) \subseteq K$. Then the map $h = g^{-1} \circ s$ contracts by $ \frac{\beta^{-1} }{m^k}$ and takes $K$ into itself, so by Theorem \ref{LCT},
\begin{equation*}
\frac{\log (| \beta^{-1}/m^k|)}{\log m} \in \mathbb{Q},
\end{equation*} 
giving $\log |\beta|/\log m \in \mathbb{Q}$, as required. \end{proof}

\subsection{Principal and non-principal projections of products}\label{sub:projections}

We turn to linear images. A principal linear functional $\mathbb{R}^{2}\rightarrow\mathbb{R}$ is a linear functional whose kernel is one of the axes. The image
of a set $X\subseteq\RR^2$ under a non-principal functional will be called a non-principal
linear image of $X$.

We say that a matrix is anti-diagonal if the only non-zero entries are on the minor diagonal (for $2\times 2$ matrices this just means the diagonal entries are $0$).

For any product set $A\times B$ let $P_2, P_1$ denote the coordinate projections, $P_2 (x,y) = y, P_1 (x,y) =x$.

\begin{Lemma}
Let $A$ be a similarity. If $P_{1}\circ A$ or $P_{2}\circ A$ are
principal functionals, then $A$ is either diagonal or anti-diagonal,
and in this case $A$ is anti-diagonal if and only if $P_{1}\circ A$ is proportional to $P_2$ and $P_{2}\circ A$ is proportional to $P_1$.\end{Lemma}
\begin{proof}
Elementary.
\end{proof}

The following result strengthens Marstrand's projection theorem to products of deleted-digit sets, asserting that \emph{every} non-principal linear image of them has dimension which is ``as large as it can possibly be'' (Marstrand's theorem only gives this for a.e. linear image).

\begin{theorem}[Peres-Shmerkin \cite{peres2009resonance}, Hochman-Shmerkin \cite{hochman2009local}] \label{Hochman - Shmerkin}
  Let $K_1 = D( \Lambda_1 ,m)$ and $K_2 = D(\Lambda_2 ,n)$ with $m,n$ multiplicatively independent. Then for any non-principal functional L,
\begin{equation}
  \dim L(K_1\times K_2) = \min (1 ,  \dim K_1 + \dim K_2).
\end{equation}
Furthermore, the same holds if $K_1\times K_2$ is replaced by $K'_1\times K'_2$, where $K'_i$ is a cylinder of $K_i$.
\end{theorem}
To derive the last statement from the first, note that by self-similarity,  $K'_i=a_iK_i + t_i$ for suitable $a_i,t_i\in\mathbb{R}$. By linearity, $L(K'_1\times K'_2)=LA(K_1\times K_2)+L(t_1,t_2)$ where $A=\diag(a_1,a_2)$, and since $A$ is diagonal, if $L$ is non-principal, so is $LA$.

\subsection{Bedford-McMullen carpets} \label{Bedford - McMullen carpets}
The main objects we will be working with are Bedford-McMullen carpets. Let $m\neq n$ be integers greater than one. We shall always assume $m>n$. Let 
\begin{equation*}
\Gamma \subseteq \lbrace 0,...,m-1 \rbrace \times \lbrace 0,...,n-1 \rbrace \;=\;[m]\times[n],
\end{equation*}
and define
\begin{equation*}
  \label{eq:BM-carpet-via-digit-expansions}
F = \lbrace (\sum_{k=1} ^\infty \frac{x_k}{m^k}, \sum_{k=1} ^\infty \frac{y_k}{n^k}) : (x_k,y_k) \in \Gamma \rbrace.
\end{equation*}
 $F$ is then called a Bedford-McMullen carpet with defining exponents $m,n$. Note that $F$ is a self affine set generated by an IFS consisting of maps whose linear parts are diagonal matrices. Specifically, $F$ is the attractor of $\Phi = \lbrace \phi_{(i,j)} \rbrace_{(i,j) \in \Gamma}$ where
\begin{equation} \label{Genrating IFS for F}
\phi_{(i,j)} (x,y) = (\frac{x+i}{m}, \frac{y+j}{n}) = \begin{pmatrix}
\frac{1}{m} & 0 \\
0 & \frac{1}{n}
\end{pmatrix}
\cdot (x,y) + (\frac{i}{m},\frac{j}{n}).
\end{equation}

Set $\Omega_{m,n} =\Omega_m \times \Omega_n \cong  ([m] \times [n])^\mathbb{N}$ with the product topology. The shift on $\Omega_{m,n}$ is $\sigma_{m,n} = \sigma_m \times \sigma_n$. Also define the projection  $\pi_{m,n}=\pi_m \times \pi_n:\Omega_{m,n} \rightarrow \mathbb{R}$. Then  $\widetilde{F} = \Gamma^\mathbb{N} \subseteq \Omega_{m,n}$ is a shift invariant subset satisfying $\pi_{m,n} (\widetilde{F}) = F$. As before, this may not be an injection, even though it is surjective, and $\widetilde{F} \subseteq \pi_{m,n} ^{-1} (F)$, but the two sets might not be equal.

For $l \in \mathbb{N}$, we write $\Gamma^l$ to denote words of length $l$ in the alphabet $\Gamma$, which we identify with pairs of words of length $l$ over $[m]$ and $[n]$ respectively. For a fixed $(b_1,...,b_l)=b \in [n]^l$ we write
\begin{equation*}
\Gamma _{b} = \lbrace a \in [m]^l :  (a , b) \in \Gamma^l \rbrace,
\end{equation*}
which  could be empty. Similarly, for $(a_1,...,a_l)=a \in [m]^l$ we write
\begin{equation*}
\Gamma ^{a} = \lbrace b \in [n]^l :  (a , b) \in \Gamma^l \rbrace.
\end{equation*}

For $y \in P_2 (F)$, define 
\begin{equation*}
F_y = \lbrace x \in \mathbb{R} : (x,y) \in F \rbrace.
\end{equation*}
We shall refer to $F_y$ as the horizontal slice of $F$ at height $y$, noting that $F_y \times \lbrace y \rbrace = F\cap (\mathbb{R} \times \lbrace y \rbrace)$. In the symbolic context, for an infinite sequence $\eta \in \Omega_n$ we define the symbolic slice corresponding to $\eta$ by
\begin{equation*}
\widetilde{F} _{\eta} = \lbrace  \omega \in \Omega_m : ( \omega , \eta ) \in \widetilde{F} \rbrace \;=\;\prod_{i=1}^\infty \Gamma_{\eta_i}.
\end{equation*}

Similarly, for $x \in P_1 (F)$ we define the vertical slice over $x$ to be
\begin{equation*}
F^x = \lbrace y \in \mathbb{R} : (x,y) \in F \rbrace.
\end{equation*}
and for an infinite sequence $\omega \in \Omega_m$ the symbolic slice corresponding to $\omega$ is 
\begin{equation*}
\widetilde{F} ^{\omega} = \lbrace  \eta \in \Omega_n : ( \omega , \eta ) \in \widetilde{F} \rbrace \;=\;\prod_{i=1}^\infty \Gamma^{\omega_i}.
\end{equation*}

Note that
\begin{equation*}
\pi_{m} ( \widetilde{F} _\eta ) \subseteq F_{\pi_m (\eta) },
\end{equation*}
but the two sets might not be equal if $\pi_n (\eta) \in [0,1]$ admits another base-$n$ expansion in $\widetilde{F}$. But we always have  that \[
  F_y = \bigcup_{\eta\in\pi_m^{-1}(y)} \pi_m(\widetilde{F}_\eta )
\]
This is a union of at most two sets (again, if one pre-image of $y$ is not in $\widetilde{F}$, the corresponding term in the union is empty). Given $\eta$, we can describe the set $\pi_m(\widetilde{F}_\eta)$ using a recursive Moran-type construction: it is the intersection $\bigcap_{k=1} ^\infty (\cup E^k)$, where $E^k$ are finite collections of closed intervals:   $E^1$ is the collection $[\frac{i}{m}, \frac{i+1}{m}], i\in \Gamma _{y_1}=\Gamma_{y_1}$, and  $E^{k}$ is obtained from $E^{k-1}$ by subdividing each interval in $E^{k-1}$ into $n$ equal closed sub-interval meeting only at endpoints, keeping those which correspond to digits in $\Gamma_{y_k}$ (with the intervals enumerated as usual from left to right), and discarding the rest.  

Consequently, Bedford-McMullen carpets admit self similar sets defined by digit restriction as horizontal slices. Indeed,  for any $j\in [n]$ such that $|\Gamma_j| \neq 0$,  for  $y=\sum_{k=1} ^\infty \frac{j}{n^k}$ the slice $F_y$ is a self similar, and in fact is equal to $D(\Gamma_j,m)$. See also Lemma \ref{lem:self-similar-slices} below.

We also have an elementary expression for the Hausdorff dimension of projections of symbolic slices: given $\eta\in \Omega_n$, 
\begin{equation}
\label{eq:slice-dimension}
   \dim \pi_m(\widetilde{F}_\eta) = \liminf_{l\to\infty} \frac{\sum_{i=1}^l \log|\Gamma_{\eta_i}|}{l\log m}
\end{equation}
(this is standard, but we will only need the even more trivial upper bound, which is obtained by using the coverings given by the Set $E^k$ above). Similarly, for $\omega\in\widetilde{F}$, \begin{equation}
\label{eq:slice-dimension vertical}
   \dim \pi_n(\widetilde{F}^\omega) = \liminf_{l\to\infty} \frac{\sum_{i=1}^l \log|\Gamma^{\omega_i}|}{l\log m}
\end{equation}

With regard to projections, note that $P_1F = D(P_1\Gamma,m)$ and $P_2F=D(P_2\Gamma,n)$. These identities can be verified directly, e.g. using \eqref{eq:BM-carpet-via-digit-expansions}.

To illustrate our discussion, consider
\begin{equation*}
F = \lbrace (\sum_{k=1} ^\infty \frac{x_k}{3^k}, \sum_{k=1} ^\infty \frac{y_k}{2^k}) : (x_k,y_k) \in \lbrace (0,0),(1,1),(2,0) \rbrace \rbrace.
\end{equation*}
That is, $F$ has defining exponents $3,2$ and $\Gamma = \lbrace (0,0),(1,1),(2,0) \rbrace \subseteq [3]\times [2]$. Then $F_0$ is just the Cantor thirds set $C=D ( \lbrace 0,2 \rbrace,3)$. For an example of a horizontal slice that is the union of two Moran sets, consider
\begin{equation*}
F = \lbrace (\sum_{k=1} ^\infty \frac{x_k}{3^k}, \sum_{k=1} ^\infty \frac{y_k}{2^k}) : (x_k,y_k) \in \lbrace (0,0), (2,0), (0,1), (1,1), (2,1) \rbrace \rbrace.
\end{equation*}
Then
\begin{equation*}
\frac{1}{2} = \frac{1}{2} + \sum_{k=2} ^\infty \frac{0}{2^k} = \frac{0}{2}+ \sum_{k=2} ^\infty \frac{1}{2^k}
\end{equation*}
so $F_{\frac{1}{2}}$ is the union of $C = \sum_{i=0} ^2 \frac{1}{3}\cdot C + \frac{i}{3}$, and $ [0,\frac{1}{3}] \cup [\frac{2}{3},1]$.

\subsection{Glossary of main notations}

We summarize our main notation in the table below. Some of it will be defined in later sections. 

\begin{tabular}{|c|l|}

\hline
  
Notation & Interpretation \\ 

\hline

%%$l, l,k,i,j,p,m,n$ & Integers, usually naturals \\ 

$g,h$ & Invertible Affine maps $g,h:\mathbb{R}^2 \rightarrow \mathbb{R}^2$, usually similarities \\

$\dim$ & Hausdorff dimension \\

$Q$ & $[-1,1]^2$ \\

$P_1, P_2$ & The  projections $\mathbb{R}^2 \rightarrow \mathbb{R}, P_2 (x,y)=y, P_1 (x,y)=x$ \\

$[m]$, $m\in \mathbb{N}$ & The set $\lbrace 0,...,m-1\rbrace$ \\ 

$D( \Lambda ,m)$ & Deleted digit set with digits in $\Lambda  \subseteq [m]$: $\{\sum_{k=1}^\infty \xi_k/m^k \;:\; \xi_k\in\Lambda\}$ \\

$\Gamma$ &  Set of digits pairs defining $F$, $\Gamma \subseteq [m]\times [n]$\\

$F$ & Bedford-McMullen carpet defined by $\Gamma$\\

$f = (f_1,f_2)$ & Element of a Bedford-McMullen carpet $F$. \\

$\Omega_m, \Omega_{m,n}$ & $\Omega_m = [m]^\mathbb{N}$, $\Omega_{m,n} = \Omega_m \times \Omega_n$ \\

$\sigma_m, \sigma_{m,n}$ & Shift operators on  $Omega_m,\Omega{m,n}$ resp.\\

%%$\omega, \eta, \xi, \tau$ & Infinite sequences in the space $\Omega_m$ or $\Omega_n$ \\

$\pi_m$ &   The "projection" $\pi_m:\Omega_m \rightarrow [0,1], \pi_m (\omega) = \sum_{k=1} ^\infty \frac{\omega_k}{m^k}$ \\

%% $\overline{\omega}, \omega\in \Omega_m$ & The elements of the set $\pi_m ^{-1} (\pi_m (\omega))$ \\

%% $\omega_1 ^L$ & A finite sequence $(\omega_1,...,\omega_l)\in [m]^l$ \\

%% Prefix of order $l$ of $\eta \in \Omega_n$ & The digits  $\eta_1 ^l \in [n]^l$ \\

%%Base $m$ expansion of $x$ &  $\omega\in \Omega_m$ such that $\pi_m (\omega)=x$, $x\in [0,1]$ \\

$\widetilde{F}$ & Symbolic version of $F$: $\widetilde{F}=\Gamma^\mathbb{N} \subseteq \Omega_{m,n}$, \\

%%$\Gamma^l$ & $\lbrace (a, b)\in [m]^l \times [n]^l: (a_i,b_i)\in \Gamma, \forall 1\leq i \leq l \rbrace$ \\

%%$I=(I_1,I_2)$ & Element of $\Gamma^l$, sometimes denoted $I^l = (I_1 ^l , I_2 ^l)$ \\

$F_y$ & Slice at height $y$: $\lbrace x\in \mathbb{R}: (x,y)\in F \rbrace$\\

$F^x$ & Slice above  $x$: $\lbrace y\in \mathbb{R}: (x,y)\in F \rbrace$\\

$\Gamma _{b}$, $b \in [n]^l$ & $\lbrace a \in [m]^l : (a , b)\in \Gamma^l \rbrace $\\

$\Gamma ^{a}$, $a \in [m]^l$ & $\lbrace b \in [n]^l : (a , b)\in \Gamma^l \rbrace $\\

$\widetilde{F} _\eta$, $\eta \in \Omega_n$ & Symbolic slice at ``height'' $\eta$: $\lbrace \omega \in \Omega_m: (\omega,\eta) \in \widetilde{F} \rbrace$ \\

$\widetilde{F} ^\omega$, $\omega \in \Omega_m$ & Symbolic slice ``above'' $\omega$: $\lbrace \eta \in \Omega_m: (\omega,\eta) \in \widetilde{F} \rbrace$ \\

$\overline{\eta}$ & For $\eta\in\Omega_n$, the other expansion of $\pi_n(\eta)$, or $\eta$ if none exists.\\

$\cmpct(Q)$ & The space of closed non-empty subsets of $Q$, with Hausdorff metric\\

$T(F,f,m,l)$ & $[m^l (F-f)]\cap Q$ \\

$T(F,f,m)$ &  Accumulation points in $\cmpct(Q)$ of $T(F,f,m,l)$, $l\rightarrow \infty$ \\

$S(\eta),\eta\in \Omega_n $ & Accumulation points of $\lbrace (\sigma_n ^l \eta, l\log_m n)\rbrace_{l \in \mathbb{N}} \in \Omega_n \times \mathbb{T}$ \\

$S'(\eta) $ & Accumulation points of $\lbrace \sigma_n ^l \eta \rbrace_{l \in \mathbb{N}}$ in $\Omega_n$ \\

%%$k=k(l), l\in \mathbb{N}$ & $\lfloor l \log_n m \rfloor$ \\

%%Cylinder set of $K$ & $\phi_{i_1} \circ ... \circ \phi_{i_k} (K), \phi_{i_j} \in \Phi$ \\
\hline
\end{tabular} 

\bigskip
Principal and non-principal projections and anti-symmetric matrices were defined in Section \ref{sub:projections}. For the terms $( \eta, s)$-set and $(\eta,s)$-multiset see Sections \ref{subsec:tangent-sets} and \ref{New notation}.

\section{Affine embeddings of deleted digits sets}\label{Section - embedding deleted digit set}

Theorem \ref{LCT} tells us that when a base-$n$ deleted digit set is embedded in itself by a similarity, the contraction ratio is a rational power of $n$. In this section we present a complementary result on the translation part.

\begin{Proposition}\label{Lemma - translation of self embedding}
Let $K = D (\Lambda,n)$ with  $2\leq |\Lambda| <n$. Let $l\in \mathbb{N}$ and suppose $\frac{1}{n^l}K + t\subseteq K$. Then $t=u/n^{l+1}$ for some integer $0\leq u<n^{l+1}$. 
\end{Proposition}

\begin{proof}  

  We first claim that we can assume that $0\in\Lambda$. For suppose we knew the proposition held in that case and suppose that $d=\min\Lambda\neq 0$. Set $z=\pi_n(d,d,d,\ldots)$ and  $K'=K-z$ and note that $K'=D(\Lambda-d,n)$ and $0\in\Lambda-d$. Thus, if $\frac{1}{n^l}K+t\subseteq K$ then \[
    \frac{1}{n^l}K'+(t -  z + \frac{1}{n^l}z) \subseteq  K' 
  \]
  We conclude that $t-z+z/n^l$ is $n$-adic rational with denominator $n^l$.  Observing that $z-z/n^l$ is $n$-adic rational with this denominator and that $t-z+z/n^l\geq 0$ (since $0\in K'$), and we conclude that $t$ has the same form, as desired.
  
We assume from now on that $0\in\Lambda$. Next, a simple remark: Let $x,y \in [0,1]$ and let $ \xi, \eta \in \Omega_n$ be base-$n$ expansions of $x,y$ respectively. If $x+y\leq 1$ then we can compute $\tau \in \Omega_n$ such that $\pi_n (\tau) = x+y$ using the usual addition with carry algorithm:
\begin{equation*}
x+y = \sum_{k=1} ^\infty \frac{\xi_k}{n^k}+ \sum_{k=1} ^\infty \frac{\eta_k}{n^k}  = \sum_{k=1} ^\infty \frac{\tau_k}{n^k},
\end{equation*}
where \[
   \tau_k = \xi_k+\eta_k+1_{\{\xi_{k+1}+\eta_{k+1}\geq n\}}\bmod 1
   \]
and $1_{\{\xi_{k+1}+\eta_{k+1}\geq n\}}$ is $1$ if a carry occurred in the digit $k+1$ and zero otherwise. 

Suppose first the $t$ admits a unique expansion $\eta$ in base $n$, that is, $|\pi_n ^{-1} (t)| =1$.  Let $\eta$ denote this expansion. We will derive from this a contradiction.

Our assumption that $0\in \Lambda$ implies that $0\in K$ hence $t=\frac{0}{n^l}+t\in K$ and it follows by the uniqueness of the above expansion that $\eta_k \in \Lambda$ for all $k\in \mathbb{N}$. 

We claim that  $\eta_{l+p} + \Lambda \mod n \subseteq \Lambda$  for every $p\in \mathbb{N}$. Otherwise,  we may find a $p$ and $d\in \Lambda$ such that $j=\eta_{l+p} + d \mod n  \notin \Lambda$. Take $x=d/n^p$, and let $\xi$ denote the expansion of $x$ with $\xi_{p} = d$ and $\xi_{k} =0$ otherwise. Then $\xi\in\Lambda^\NN$ so  $x=\pi_n(\xi)\in K$, so also $y=x/n^l+t\in K$, and $y$ admits a unique base-$n$ expansion $\tau$, since $t$ does and $x$ is $n$-adic rational. This expansion is in $\Lambda^n$ since $y\in K$, and computing $\tau$ (noting that no carries occur in digits greater than $p$) we have
\begin{equation*}
\tau _{l+p} = d+ \eta_{l+p} \mod n = j \notin \Lambda.
\end{equation*}
which is a contradiction.

Next, we claim that $\eta_{l+2} =0$. Suppose the contrary is true; we will show that $\eta_{l+1}+\Lambda+1\subseteq\Lambda\bmod 1$. We have already seen above that  $\eta_{l+2} +\Lambda \mod n \subseteq \Lambda$, and  it follows from this and the assumption $\eta _{l+2} \neq 0$ that we can find a $d\in \Lambda$ such that $\eta_{l+2} + d \geq n$; for example take  $d=k\eta_{l+2}$, where $k$ is the greatest integer such that $k\eta_{l+2}<n$. Now fix $j\in\Lambda$ and let $x=j/n+d/n^2$, which has the expansion $\xi=(j,d,0,0,\ldots)\in\Lambda^\NN$, so belongs to $K$. Then  $y=x/n^l+t\in K$ has a unique expansion $\tau$ which we can compute by addition with carry using $\eta$ and $\xi$, and we find that a carry is generated in the $l+2$-nd digit, hence the $l+1$-st digit of $\tau$ is
\begin{equation*}
\tau_{l+1} = j+ \eta_{l+1}+1 \mod n
\end{equation*}
Since $y=\pi_n(\tau)$ belongs to $K$ and $\tau$ is its unique expansion we have $j+\eta_{l+1}+1\in\Lambda\bmod 1$, as claimed. Thus, we now have $\eta_{l+1}+\Lambda \subseteq\Lambda\bmod 1$ and  $\eta_{l+1}+\Lambda + 1 \subseteq\Lambda\bmod 1$. This means that $\Lambda$ contains the additive sub-semigroup of $\ZZ/n\ZZ$ generated by $\eta_{l+1}$ and $\eta_{l+1}+1\bmod 1$, which implies that $\Lambda=[n]$, contrary to our assumption  that $|\Lambda|<n$. We conclude that $\eta _{l+2} =0$.

Finally, apply the argument above inductively to show for each integer $p>2$ that $\eta _{l+p } \neq 0$, and we conclude that $\eta_k =0$ for all $k>l+1$. This contradicts the assumption that $t$ has a unique expansion.

Now, suppose that $t$ admits two expansions in base $n$. They may not both be in $\Lambda^\NN$, but since $t\in K$ at least one of them is. Denote it $\eta$. We cannot apply verbatim the argument from the case when $t$ had unique expansion, because now, if we would choose $\xi\in\Lambda^\NN$ with  $x =\pi_n(\xi)\in K$ an $n$-adic rational, then $y=x/n^l+t$ will be $n$-adic rational too, and it may be  that the expansion we get from performing addition with carry on $\eta$ and $\xi$ is not the expansion of $y$ that is in $\Lambda^\NN$. However, the proof works with the following modification: instead of setting $\xi$ to have a tail of $0$'s, let $0\neq d\in\Lambda$ (which exists by our assumption $|\Lambda|\geq 2$), and let $\xi$ terminate with $0,d,0,d,0,d,\ldots$ instead. Then neither $x$ nor $x/n^l$ is $n$-adic rational, and since $t$ is, $y=x/n^l+t$ has a unique expansion. Also, though the tail may generate carries, the first $0$ in the tail sequence prevents carries from propogating to the digits we are interested in. Thus, the argument goes through.
\end{proof}

\begin{Corollary}\label{Cor:generalized-embedding-of-dd-sets}
  Let $K=D(\Lambda,n)$ be such that $0<\dim K<1$, let $g(x)=n^{-l}x+t$ for some $l\in\mathbb{N}$, and suppose that \[
  g(K)\subseteq \bigcup_{i=1}^\infty (n^{-l_i}(K+p_i))
\] for integers $l_i,p_i$. Then $t=u/n^k$ for some $k\in\mathbb{N}$ and $0\leq u<n^k$.
\end{Corollary}
\begin{proof}
  By Baire's theorem there is an open set in $g(K)$ that is contained in one of the sets $K'=n^{-l_i}K+p_i$. Thus there is a cylinder set $K''\subseteq K$ with $g(K'')\subseteq K'$. There exist $k,q$ such that $K''=n^{-k}(K+q)$. Thus $hK\subseteq K$, where $h=h_1^{-1}gh_2$ for $h_1(x)=n^{-l_i}(x+p_i)$ and $h_2(x)=n^{-k}(x+q)$. Then $h$ contracts by $n^{k-l-l_i}$ and translates by an amount which differs from $t$ by an $n$-adic rational. Applying Proposition \ref{Lemma - translation of self embedding}, this translation, and hence $t$, is $n$-adic rational.

\end{proof}

\section{$m$-adic tangent sets of McMullen carpets}\label{subsec:tangent-sets}

This section introduces the definition and basic results on $m$-adic tangent sets, which will play a role in the proof of our main theorem.

\subsection{The Hausdorff metric}
Denote
\begin{equation*}
Q=[-1,1]^2 .
\end{equation*}
and write $\cmpct(Q)$ for the set of non-empty closed subsets of $Q$. For $A,B\in\cmpct(Q)$ and $\epsilon >0$ define, using the Euclidean norm $|| \cdot ||$,
\begin{equation*}
A_\epsilon = \lbrace x \in Q : \exists a\in A, || x-a|| < \epsilon \rbrace.
\end{equation*}
The Hausdorff distance is defined by
\begin{equation*}
d_H (A,B) = \inf \lbrace \epsilon >0 : A \subseteq B_\epsilon, B \subseteq A_\epsilon \rbrace. 
\end{equation*}
This is a compact metric on  $\cmpct(Q)$ (see e.g. the appendix in \cite{bishop2013fractal}). We state without proof some of its basic properties:

\begin{Proposition} \label{Proposition - Hausdorff metric0} Let $X_n,X\in\cmpct(Q)$ and $X_n\to X$.
\begin{enumerate}
\item \label{H-limit-pointwise}The limit $X$ is given by
\begin{equation*}
X = \lbrace x\in Q : \exists x_{n_k} \in X_{n_k} \text{ such that } \lim x_{n_k} = x\rbrace.
\end{equation*}

\item \label{H-limit-of-union}Let $l\in\NN$ and suppose $X_n = \bigcup_{k=1} ^l X_n ^k$ with $X_n^k\in\cmpct(Q)$.  If for each $k$ we have $X_n ^k \rightarrow X^k$ as $n\rightarrow \infty$, then $X=\lim X_n= \bigcup_{k=1} ^l X^k$.

\item \label{H-limit-on-open-set} Let $Y\subseteq Q$ be closed with (possibly empty) interior $Y^{\circ}$, and suppose that $X_{i}\cap Y\rightarrow X'$. Then 
\[
X\cap Y^{\circ}\subseteq X'\subseteq X\cap Y
\]

\item \label{H-limit-of-image}Suppose $f:Q\to Q$ is continuous. Then $f(X_n),f(X)\in\cmpct(Q)$ and $f(X_n)\to f(X)$. 

\end{enumerate}

\end{Proposition}

\subsection{$m$-adic tangent sets of Bedford-McMullen carpets}

Let $F\in\cmpct(Q)$, fix $f\in F$ and let $l\in \mathbb{N}$. Define the $m$-adic mini-set of $F$ at $f$ by
\begin{equation*}
T (F , f, m, l) =[ m^l ( F -f)]\cap Q \in \cmpct(Q).
\end{equation*}
(note that  $T(F,f,m,l)\neq\emptyset$ because it contains $0$). The set of $m$-adic tangent sets of $F$ is  defined  to be the set of all accumulation points of $T(F,f,m,l)$ in the Hausdorff metric as we take $l\rightarrow \infty$, that is,
\begin{equation*}
T(F,f,m) = \lbrace T \in \cmpct (Q) :   \exists \lbrace l_k \rbrace \subseteq \mathbb{N}, T = \lim_{k\to\infty} T(F,f,m,l_k) \rbrace .
\end{equation*}
This collection of sets captures information about the microscopic structure of$F$ as we ``zoom in'' to $f$.

The proof of Theorem \ref{Main Theorem} relies on a structure theorem for the tangent sets when $F$ if a Bedford-McMullen carpet with multiplicatively independent defining exponents $m>n$, and from now on we assume it is such a set. For $\eta\in\Omega_{n}$  and $s\in [0,1)$ we define an $(\eta,s)$-set to be a set of the form
\begin{equation} \label{Location}
\left(\begin{array}{cc}
1\\
 & n^s
\end{array}\right)\cdot\left(\pi_{m}(\widetilde{F}_{\eta})\times P_{2}(F)+z\right)
\end{equation}
which is contained in $[-2,2]^{2}$.  We say that a set
$E\subseteq Q$ is a $(\eta,s)$-multiset if there are finitely many
$(\eta,s)$-sets $E_{1},\ldots,E_{N}$ 
such that 
\begin{equation}\label{eq:2}
 \bigcup_{i=1}^{N}E_{i}\cap(-1,1)^{2}\subseteq E\subseteq\bigcup_{i=1}^{N}E_{i}\cap[-1,1]^{2}
\end{equation}
Finally, for $\eta \in \Omega_n$ let $S(\eta)\subseteq \Omega_n\times \mathbb{T}$
denote the set \begin{equation}
  \label{eq:Su}
  S(\eta)=\{(\xi,s)\in\Omega_n\times \mathbb{T}\,:\,(\sigma_n^{l_k}\eta,l_k \log_nm)\to (\xi,s) \textrm{ for some }l_k\to\infty\}
\end{equation}
i.e. $S(\eta)$ is the set of accumulation points of the orbit of $(\eta,0)$
under the transformation $(\xi,s)\mapsto(\sigma_n\xi,s+\log_nm)$. We also define $S'(\eta)\subseteq\Omega_n$ to be the set of accumulation points of the orbit of $\eta$ under $\sigma_n$, so \begin{equation}
  \label{eq:Sprimeu}
  S'(\eta)=P_1S(\eta)
\end{equation}

For $\eta\in\Omega_n$, let $\overline{\eta}=\eta$ if $\pi_n(\eta)$ has a unique base-$n$ expansion, and  otherwise let $\overline{\eta}$ be the other expansion.

We defer the proof of the following theorem to Section \ref{Section - proof of Theorem tangent sets}.

\begin{theorem} \label{Theorem - structure of tangnet sets}
Fix $f=(f_{1},f_{2})\in F$ with $f_2\neq 0,1$ and let $\eta\in\pi_n^{-1}(f_2)$. Then for every $m$-adic tangent set
$T\in T(F,f,m)$, there exists $(\xi,s)\in S(\eta)$ such that $T$ is a non-empty union of a $(\xi,s)$-multiset and a $(\overline{\xi},s)$-multiset. Conversely, if $(\xi,s)\in S(f_{2})$, then there is an
$m$-adic tangent set $T\in T(F,f,m)$ which a union of this type.

In the special case when $f_2=0$ or $f_2=1$, the same is true but omitting the $(\overline{\xi},s)$-multiset from the union.
\end{theorem}

We emphasize that a tangent set always contains $0$. Therefore the $(\xi,s)$-multisets in the theorem intersect $(-1,1)^2$ non-trivially, and hence $T\cap (-1,1)^2$ contains a non-trivial open subset of a $(\xi,s)$-set.

The case $f_2=0,1$ is treated differently because, unlike other $n$-adic rationals, these numbers have a unique base-$n$ expansion. The geometric significance of this is that if  $f_2$ has two expansions, then a small square around $f$ can intersect $F$ at points both above and below the horizontal line bisecting the square, at points whose expansions end in both zeros and $n-1$'s. In contrast, for $f_2=0,1$, a small square around $f$ is centered at the top slice or bottom slice of $F$ and  as we ``zoom in'' along small squares centered at $f$, it remains the case that either the upper or bottom half of the set is empty.

In applications, we shall either not care about the identity of the limit point $(\xi,s$), provided in the theorem, or else we will control it by starting with  $f_2$  whose expansions are suitably engineered. The second component of $S(y)$ will not play any role in our analysis, but we have included it for future use.

\subsection{Covariance of tangent sets under affine embedding}\label{subsec:sovariance-of-tangent-sets} 

In our analysis we will use the fact that tangent sets transform nicely under affine embeddings (and more generally diffeomorphisms, but we shall not need this here). Specifically,

\begin{Proposition} \label{Proposition - Hausdorff metric}
  Let $g(x)=Ax+t$ be a non-singular affine transformation of $\mathbb{R}^2$. Let $X\subseteq Q$ be closed, and suppose that $g(X)\subseteq X$. Let $x\in X$ and set $y=g(x)\in X$.   For $p\in\NN$ large enough that $m^{-p}AQ\subset Q$, if  $T = \lim_{k\to\infty} [m^{l_k} (X -x])\cap Q \in T(X,x,m)$ and $T' = \lim_{k\to\infty}[m^{l_k-p} (X-y)]\cap Q \in T(X,y,m)$ as $k\to\infty$, then $m^{-p}AT \subseteq T'$.
\end{Proposition}

\begin{proof}
First assume that $p=0$. Then for every $l\in \mathbb{N}$ we have 
\begin{eqnarray*}
  [m^l (X-y)]\cap Q &=& [m^l (X-(Ax+t))]\cap Q \\
   &\supseteq& [m^l (AX+t-(Ax+t))]\cap Q \\
  &\supseteq& [m^l (AX-Ax))]\cap AQ\\
  &=& [m^l (A(X-x))]\cap AQ\\
  &=& [A( m^l (X-x))]\cap AQ\\
  &=& A( [m^l (X-x)]\cap Q).
\end{eqnarray*}
Let us  justify this calculation: the first equality follows since $y= Ax+t$. The second inclusion follows since $AX + t \subseteq X$. The third inclusion follows since $AQ \subseteq Q$. The fourth equality follows since $A$ is a linear map. The fifth equality follows since $A$ commutes with scalars. The final equality is true because $A$ is by assumption non-singular.

Now, by Proposition \ref{Proposition - Hausdorff metric0}\eqref{H-limit-of-image}, $[m^{l_k} (X-x)]\cap Q \rightarrow T$ implies $A ([m^{l_k} (X-x)]\cap Q) \rightarrow A(T)$. But by the above, $A ([m^{l_k} (X-x)]\cap Q)\subseteq [m^{l_k}(X-y)]\cap Q$, so by Proposition \ref{Proposition - Hausdorff metric0}\eqref{H-limit-on-open-set} (taking $Y=Q$), we obtain the desired relation between $AT \subseteq T'$.

The case $p>0$ is proved entirely analogously, starting from $[m^{l-p}A(X-y)]\cap Q\subseteq [m^{-l}(X-y)]\cap Q$ and using the fact that $m^{-p}AQ\subseteq Q$. We omit the details.
\end{proof}

To generalize this to the case of a diffeomorphisms $g$, replace $A$ by the derivative $D_xg$ in the conclusion, requiring it to be non-singular, assume $p$ large enough with respect to this map, and use the fact that $g|_{m^{-l}Q+x}(z)=D_xg(z)+o(|z|)$.

\section{Product case: Proof of Theorem \ref{Theorem product sets}}\label{Section - proof of product set theorem}
The following proof illustrates how one can apply the projection Theorem \ref{Hochman - Shmerkin} in order to study self-embedding of  products of (non-trivial) deleted digit sets. This  idea will play a key role also in the non-product case in the next section.

Let $K_1,K_2$ denote deleted digit sets in multiplicatively independent bases $m,n$, respectively, and assume their dimensions are strictly between $0$ and $1$. Let $g(x)=Ax+t, A\in GL(\mathbb{R}^2)$, and suppose that $g(F)\subseteq F$. We wish to show that $A=\diag(m^{-p},n^{-q})$ for some $p,q\in\mathbb{Q}$.

If $A$ is not diagonal or anti-diagonal, then $P_1\circ g$ and $P_2\circ g$ are non-principal. Thus $P_i\circ g(F)$ are non-principal images of $F=K_1\times K_2$, and they are subsets of $K_i =P_i(F)$ respectively. We conclude from Theorem \ref{Hochman - Shmerkin} that \begin{eqnarray*}
  \dim K_1 &\geq &\min\{1,\dim K_1 + \dim K_2\}\\
  \dim K_2 &\geq &\min\{1,\dim K_1 + \dim K_2\}\\
\end{eqnarray*}
This is possible only if $\dim K_1=\dim K_2 =1$ or $\dim K_1=\dim K_2 =0$, neither of which is consistent with our assumptions.

Thus $A$ is either diagonal or anti-diagonal. To rule out the latter possibility, note that if $A$ were anti-diagonal, then $P_1\circ g$ and $P_2\circ g$ would, respectively, be  affine embeddings of $K_2 \to K_1$ and of $K_1\to K_2$. By Theorem \ref{LCT2} this is only possible if these are singular affine maps, which is again a contradiction.

Finally, since $A$ is diagonal, we have that $P_1\circ g$ maps $K_1\to K_1$, so by Theorem \ref{LCT} its contraction ratio is a rational power of $m$, giving the first entry on the diagonal of $A$. The second entry is obtained by considering $P_2\circ g$. 

\section{Non-product case: Proof of Theorem \ref{Main Theorem} } \label{Section - proof of main Theorem}

Throughout this section we fix the following notation. Let $F$ be a Bedford-McMullen carpet defined by multiplicatively independent exponents $m>n\geq2$ and assume that $F$ is not a product set (we dealt with that case in the previous section). Let $g(x)=\alpha Ox + t$ be a non-singular similarity with $gF\subseteq F$.

Our aim is to show that $\alpha=1$ and the linear part of $g$ has the form $\diag(\pm 1,\pm 1)$. We do this in two steps. First, in Section  \ref{subsec:reduction-to-homotheties}, we show that the linear part of $g$ is of the form $\alpha\diag(\pm 1,\pm1)$. We then show that $\alpha=1$ by assuming that $0<\alpha<1$ and deriving a contradiction. This part takes up Sections \ref{subsec:analysis-of-slices} through \ref{subsec:case-B}

\subsection{Reduction to homotheties}\label{subsec:reduction-to-homotheties}

In this section we show that the linear part of  $g$ is diagonal, i.e. of the form $\alpha\diag(\pm 1,\pm 1)$. We aply similar ideas to those in the  product set case, using the fact that the $m$-adic tangent sets of $F$ are product sets and inherit some of the symmetries of $F$.

We define a piece of a set $A$ to be a non-empty open subset of $A$,
or equivalently, a non-trivial intersection of $A$ with an open set.

\begin{Lemma}\label{non-principal-projections-in-tangent-sets}
Let $f=(f_{1},f_{2})\in F$, set $h=(h_{1},h_{2})=g(f)$, and let $T\in T(F,h,m)$ and  $\eta \in \pi_n ^{-1} (f_2)$. Then
\begin{enumerate}
\item If $O$ is neither diagonal nor anti-diagonal, then there exists a 
$\xi\in S'(\eta)$ and $\zeta\in\{\xi,\overline{\xi}\}$ such that each of the sets $P_{1}T$
and $P_{2}T$ contains a non-principal linear image of a piece of
$(\pi_{m}\widetilde{F}_{\zeta})\times P_{2}F$, and in particular they contain an
affine copy of $P_{2}F$ and an affine image of a piece of $\pi_{m}\widetilde{F}_{\zeta}$.

\item If $O$ is anti-diagonal, then there exists a $\xi\in S'(\eta)$ and $\zeta\in\{\xi,\overline{\xi}\}$
  such that  $P_{1}T$ contains an affine image of $P_{2}F$, and $P_{2}T$ contains an affine image of a piece 
of $\pi_{m}\widetilde{F}_{\zeta}$.
\end{enumerate}
\end{Lemma}
\begin{proof}
We prove (1). Suppose that $T=\lim_{k\rightarrow\infty}(m^{l_{k}}(F-h))\cap Q$
for a sequence $l_{k}\rightarrow\infty$. Passing to a further
subsequence, we can assume that $T'=\lim_{k\rightarrow\infty}(m^{l_{k}}(F-f))\cap Q$
exists, so $T'\in T(F,f,m)$. By Proposition \ref{Proposition - Hausdorff metric}, we know that 
\begin{equation}
\alpha OT'\subseteq T\label{eq:embedding-of-tangent-sets}
\end{equation}
(This is a simplification - the proposition requires that $\alpha O(Q)\subseteq Q$, and if this does not hold the conclusion is $m^{-p} \alpha O (T') \subseteq T$ for some $p\in\NN$. But the remainder of the proof proceeds unchanged so we remain with the simple version). By Theorem \ref{Theorem - structure of tangnet sets},  we can find $\xi \in S'(\eta)$ and $r\in[1,n]$ such
that $T'$ contains a homothetic image of a piece of $(\pi_{m}\widetilde{F}_{\xi})\times(r P_{2}F)$ or of $(\pi_{m}\widetilde{F}_{\overline{\xi}})\times(r P_{2}F)$. Without the loss of generality assume this is true for $\xi$. 
So, by (\ref{eq:embedding-of-tangent-sets}), $T$ contains a homothetic
image of a piece of $ O((\pi_{m}\widetilde{F}_{\xi})\times(r P_{2}F))$.
Since by assumption $O$ does not preserve the union of the axes, $\alpha P_{1} O$
and $\alpha P_{2} O$ are non-principal functionals, and
we conclude that $P_{1}T$ and $P_{2}T$ contain non-principal linear
images of pieces of $(\pi_{m}\widetilde{F}_{\xi})\times(r P_{2}F)$,
and (using a different non-principal functional) also of $(\pi_{m}\widetilde{F}_{\xi})\times P_{2}F$. This
proves the first part of the first statement. The second part follows
(note that \emph{a-priori} we only get that $P_{1}T$  contains
an affine image of a piece of $P_{2}F$, but since $P_{2}F$ is self-similar,
such a set must contain an affine image of $P_{2}F$ itself).

The proof of (2) is identical, noting that when $O$ is anti-diagonal,
$P_{1} O$ is proportional to $P_{2}$ and $P_{2} O$ is
proportional to $P_{1}$.\end{proof}

\begin{Lemma}\label{dim-of-principal-projections-of-tangent-sets}
Let $j\in[n]$ maximize $|\Gamma_{j}|$. If $f\in F$ and  $T\in T(F,f,m)$, then $\dim P_{2}T=\dim P_{2}F$, and $\dim P_{1}T\leq\log|\Gamma_{j}|/\log m$.\end{Lemma}
\begin{proof}
By Theorem \ref{Theorem - structure of tangnet sets}, $T$ contains a piece of a (translate of a) set of the form $\pi_m \widetilde{F}_{\eta}\times(r P_{2}F)$, and is contained in a finite union of sets of this form. So (i) $P_2T$ is contained in a finite union of affine images of $P_{2}F$ and (ii) $P_{2}T$ contains a piece of an affine image of $P_2 (F)$.  (i) implies that $\dim P_2T\leq \dim P_2F$. On the other hand, (ii) and self-similarity of $F$ implies that $P_2T$ actually contains an affine image of the entire set $P_2F$, so $\dim P_2T\geq\dim P_2F$, and we obtain the first equality.

For the second equality,  recall that by \eqref{eq:slice-dimension},
for any $\xi\in\Omega_{n}$,
\[
\dim\pi_{m}\widetilde{F}_{\xi}=\liminf_{k\rightarrow\infty}\frac{1}{k\log m}\sum_{i=1}^{k}\log|\Gamma_{\xi_{i}}|
\]
and since $|\Gamma_{\xi_{i}}|\leq|\Gamma_{j}|$ for all $i$ we have
$\dim\pi_{m}\widetilde{F}_{\xi}\leq\log|\Gamma_{j}|/\log m$. Since
$P_{1}T$ is contained in a finite union of homothetic images of sets of this form, we get the second dimension bound.\end{proof}

\begin{theorem}\label{thm:reduction-to-homotheties}
If $g$ is a similarity map with $gF\subseteq F$ then the linear
part of $g$ is diagonal.\end{theorem}
\begin{proof}
Assume for the sake of contradiction that the linear part of $g$
is not diagonal.

\textbf{Case~1:~$\dim P_{2}F<1$~or~$\dim F_{y}<1$~for~all~$y$}.\textbf{
}Choose $j\in[n]$ such that $|\Gamma_{j}|$ is maximal, set $\xi=(j,j,j,\ldots)\in\Omega_{n}$,
choose any $\zeta\in\Gamma_{j}^{\mathbb{N}}\subseteq\widetilde{F}_{\xi}$,
and let $f_{1}=\pi_{m}\zeta$ and $f_{2}=\pi_{n}\xi$ so that $f=(f_{1},f_{2})\in F$.
Set $h=g(f)$ and choose any $T\in T(F,h,m)$. 

Assume for the moment also that $O$ is not anti-diagonal (so $O$ is not diagonal and not anti-diagonal). Since $S'(\xi)= \lbrace \xi \rbrace$ and $\overline{\xi} = \xi$,
Lemma \ref{non-principal-projections-in-tangent-sets} tells us that $P_{1}T$ and $P_{2}T$ each contains
a non-principal linear image of a piece of $\pi_{m}(\widetilde{F}_{\xi})\times P_{2}F$,
which by Theorem \ref{Hochman - Shmerkin} and Lemma \ref{dim-of-principal-projections-of-tangent-sets} implies  
\begin{eqnarray*}
\dim P_{2}F\;=\;\dim P_{2}T & \geq & \min\{1,\dim\pi_{m}\widetilde{F}_{\xi}+\dim P_{2}F\}\\
\dim\pi_{m}\widetilde{F}_{\xi}\;\geq\;\dim P_{1}T & \geq & \min\{1,\dim\pi_{m}\widetilde{F}_{\xi}+\dim P_{2}F\}
\end{eqnarray*}
But these inequalities are possible only if $\dim P_{2}F=\dim\pi_{m}\widetilde{F}_{\xi}=0$
or $\dim P_{2}F=\dim\pi_{m}\widetilde{F}_{\xi}=1$, both of which contradict our
assumptions. 

It remains to rule out the possibility that $O$ is anti-diagonal.
If it were, we argue as above, and this time Lemma \ref{non-principal-projections-in-tangent-sets} tells
us that $P_{1}T$ contains an affine copy of a piece of $P_2 (F)$,
and hence (by self-similarity of $P_2 (F)$) an affine copy of $P_2 (F)$. But $P_{1}T$ is contained in a finite union of affine images of pieces of $\pi_{m}\widetilde{F}_{\xi}$, so by Baire's theorem and self-similarity of $\pi_{m}\widetilde{F}_{\xi}$ and $P_2 (F)$, we find that $\pi_{m}\widetilde{F}_{\xi}$  contains an affine
copy of $P_{2}F$. By Theorem \ref{LCT2}, since $\dim P_2 (F) >0$, this is only possible if $\pi_{m}\widetilde{F}_{\xi} =[0,1]$ (i.e. it does not have intermediate dimension). Thus there exist $y\in P_2 (F)$ such that $\dim F_y =1$. Since $O$ is anti-diagonal this implies that $P_2 (g(F))$ contains and affine image of $F_y = [0,1]$, so $\dim P_2 (F) =1$, contradicting our assumptions.

\textbf{Case~2:~$\dim P_{2}F=1$~and~$\dim F_{y}=1$~for~some~$y$}.
The first assumption implies that $P_2F=[0,1]$ (because $P_2F$ is a deleted-digit set), and the second that there is a $j\in[n]$ with $|\Gamma_{j}|=m$ (since otherwise, by equation \eqref{eq:slice-dimension} $\log(m-1)/\log m$ would be an upper bound on the dimension of all horizontal slices). In particular for $y=\pi_n(j,j,j,\ldots)$ we have $F_y=[0,1]$. Assuming that $O$ is not diagonal, it maps $F_y\times \{y\}$ to a line segment not parallel to the $x$-axis, so $P_{2}gF$ contains an interval.
Now, there must also be at
least one $u\in[n]$ with $|\Gamma_{u}|<m$, for if $\Gamma_{u}=[m]$
for all $u\in[n]$ then we would have $\Gamma=[m]\times[n]$ and $F=[0,1]^2$, contrary to assumption. Fix
such a $u$. Since $P_{2}F=[0,1]$ we also \,have $|\Gamma_{u}|>0$
(otherwise, $P_{2}F\subseteq D([n]\setminus\{u\},n)$ contradicting
$P_{2}F=[0,1]$). Consider the set 
\[
E=\{\xi\in\Omega_{n}\,:\,\exists k\mbox{ s.t. }\xi_{i}=u\mbox{ for all }i\geq k\}
\]
Then $\pi_{n}E$ is dense in $[0,1]$, and in particular there exists
an $h_{2}\in\pi_{n}E\cap P_{2}gF$. Choose $h_{1}$ so that $h=(h_{1},h_{2})\in gF\subseteq F$.
Choose some $T\in T(F,h,m)$. 

If $O$ is neither diagonal nor anti-diagonal, then by Lemma \ref{non-principal-projections-in-tangent-sets},
$P_{1}T$ contains an affine copy of $P_{2}F=[0,1]$, so $\dim P_{1}T=1$.
On the other hand, the set $T$ is contained in finitely many product
sets of the form $(\pi_{m}\widetilde{F}_{\zeta})\times(r P_{2}F)$
with $\zeta$ or $\overline{\zeta}$ in the orbit closure of some $\xi\in E$. Since the
orbit closure of $\xi$ consists only of constant sequences $(u,u,u,\ldots) = \overline{(u,u,u,\ldots)}$,
and $\pi_{m}\widetilde{F}_{(u,u,\ldots)}=D(\Gamma_{u},m)$ has dimension
$<1$, we conclude that $\dim P_{1}T<1$, a contradiction. 

Finally, to show that $O$ cannot be anti-diagonal, argue as above:
Lemma \ref{non-principal-projections-in-tangent-sets} implies again that $\pi_{m}\widetilde{F}_{(u,u,\ldots)}$
contains an affine copy of a piece of $P_{2}F$, which is an interval.
This contradicts $\dim\pi_{m}\widetilde{F}_{(u,u,\ldots)}<1$.\end{proof}

Having shown that the linear part of $g$ is $\alpha\diag (\pm 1,\pm 1)$, the linear part of $g^2$ is just $\alpha^2 I$, so $g^2$ is a homothety which  contracts by $\alpha^2$. Thus, if we show that $g^2$ is an isometry, then we will know that $g$ is as well. Since $g^2(F)\subseteq F$, we have reduced the proof of Theorem \ref{Main Theorem} to the following: Show that if a homothety maps $F$ into itself, then it is an isometry.

\subsection{Analysis of slices}\label{subsec:analysis-of-slices}

By the discussion at the end of the last section, it is enough to prove Theorem \ref{Main Theorem} when $g$ is a homothety. So assume henceforth that $g(x)=\alpha x+t$, and assume that $\alpha<1$; we will eventually arrive at a contradiction, implying that $\alpha=1$.

In this section we prove a technical result, showing that $\alpha<1$ implies that  there are large vertical and horizontal slices - in fact, for the horizontal case, there are ``maximally large'' slices. The proof relies again on the machinery of tangent sets.

The map $g:F\to F$ is a contraction, and $F$ is compact; therefore there exists a fixed point in $F$. We denote it by $f^*=(f^*_1,f^*_2)$, so we have $g(f^*)=f^*\in F$. 

\begin{Lemma}
  Let $X_i,X\in\cmpct(Q)$ with $X_i\to X$, let $Y\in\cmpct(Q)$, and let  $\{g_i\}$ be a bounded family of homotheties such that $g_i(Y)\subseteq X_i$. Then there exists an accumulation point $h$ of $\{g_i\}$ such that $h(Y)\subseteq X$.
\end{Lemma}
\begin{proof}
  Passing to a subsequence we can assume $g_i\to h$ for some $h$. Then $g_i(Y)\to h(Y)$, and since $g_i(Y)\subseteq X_i$, it follows easily that $h(Y)\subseteq \lim X_i = X$.
\end{proof}

\begin{Proposition}
Suppose that $\alpha<1$. Then any $m$-adic tangent set $T\in T(F,f^*,m)$ contains a homothetic image of $F$.
\end{Proposition}

\begin{proof}
Since $\alpha<1$, for each $k\in\mathbb{N}$ we can choose  $l=l(k) \in\mathbb{N}$ such that $m^{-(l+1)}<\alpha^k \leq m^{-l}$. Then $g^k([0,1]^2)$ is a cube of side $\alpha^k$ containing $f^*$ (because $g^k (f^*)=f^*$), so $g^k(F)\subseteq m^{-l}Q+f^*$. Therefore, the set $T(F,f^*,m,l)=m^l(F-f^*)\cap Q$ contains the set \[
m^l(g^k(F)-f^*) = m^l\alpha^k F + t'
\]
for some $t'=t'(l)\in\mathbb{R}^2$. Now, by choice of $l=l(k)$ we have \[
  \frac{1}{m} \leq m^l\alpha^k \leq 1
\]
In particular $m^l\alpha^kF$ is contained in a cube of side length at most $1$, and since after translation by $t'$ it is contained in $Q$ we must have $t'\in [-2,2]^2$. Thus, we see that $T(F,f^*,m,l(k))$ contains $g_{l(k)}(F)$ where $g_{l(k)}$ is a similarity with contraction in the range $[1/m,1]$, its orthogonal part is diagonal, and its translation part is in $[-2,2]^2$. The conclusion now follows directly from the previous lemma.
\end{proof}
     
\begin{Proposition} \label{Proposition 3}
If $\alpha<1$ then there exists $y\in P_2F$ such that $\dim F_y = \dim P_1 (F)$ and $x\in P_1F$ with $\dim F^x>0$.
\end{Proposition}

\begin{proof}
Consider an $m$-adic tangent set $T\in T(F,f^*,m)$. Let $\eta \in \pi_n ^{-1} (f_2 ^*)$. On the one hand, by Theorem \ref{Theorem - structure of tangnet sets} and the fact that every piece of a slice has the same dimension as the slice itself, there exists $\xi \in S'(\eta)\cup\overline{S'(\eta)}$  such that $\dim P_1T=\dim \pi_m\widetilde{F}_\xi$. On the other hand, by the previous proposition, $T$ contains the image of $F$ under a similarity whose linear part is diagonal, hence $P_1T$ contains an affine copy of $P_1F$. It follows that \[
  \dim P_1F \leq \dim \pi_m\widetilde{F}_\xi \leq \dim F_{\pi_n\xi}  \leq \dim P_1F
\]
where the last inequality is because $P_1F$ contains every horizontal slice of $F$. This proves claim about horizontal slices for $y=\pi_n(\xi)$.

For the second statement, suppose $\dim F^x=0$ for all $x$. Then  $|\Gamma^i|=1$ for all $i\in P_1\Gamma$, that is, for every $i \in P_1\Gamma$ there exists a unique $j\in [n]$ with $(i,j)\in\Gamma$. Now, by the first part there is a horizontal slice $F_y$ with dimension equal to the dimension of $P_1F=D(P_1\Gamma,m)$. It follows that there is some $j\in[n]$ such that $|\Gamma_j|=|P_1\Gamma|$, and since $\Gamma_j \subseteq P_1 \Gamma$ we have $\Gamma_j = P_1 \Gamma$. But this and the previous property of $\Gamma$ immediately imply that $\Gamma = P_1\Gamma\times\{j\}$, which in turn implies that $F$ is contained in the horizontal line of height $\pi_n(j,j,j,\ldots)$. This contradicts our standing assumptions.
\end{proof}

We end this section with two basic observations about  slices whose position is suitably rational.

\begin{Lemma}\label{lem:self-similar-slices}
  Suppose that $\eta\in\Omega_n$ and that the sequence $(\eta_k,\eta_{k+1}\ldots)$ is periodic with least period $p$. Suppose that $\eta$ is the unique expansion of $y=\pi_n(\eta)$. Then, writing $\Lambda = m^p\cdot \pi_m(\prod_{i=k}^{k+p-1}\Gamma_{\eta_i})$ and $E=\pi_m(\prod_{i=1}^{k-1}\Gamma_{\eta_i})$, \[
  F_y = E  + \frac{1}{m^{k-1}}D(\Lambda,m^p)
  \]
  In particular, if $\eta=(j,j,j,\ldots)$ then $F_y=D(\Gamma_j,m)$.
\end{Lemma}
\begin{proof}
  Let $a=\eta_1\ldots\eta_{k-1}$ and  $b=\eta_k\ldots\eta_{k+p-1}\in[n]^p$, so $\eta=abbbb\ldots$. Since $\eta$ is the unique expansion of $y$ we have \begin{eqnarray*}
F_y &=& \pi_m(\widetilde{F}_\eta)\\
&=& \pi_m(\Gamma_a\times(\Gamma_b)^{\mathbb{N}})\\
&=& \pi_m(\Gamma_a) + \frac{1}{m^{k-1}}\pi_m(\Gamma_b^{\mathbb{N}})\\
&=& E + \frac{1}{m^{k-1}}\pi_m(\Gamma_b^{\mathbb{N}})\\
  \end{eqnarray*}
  The claim follows upon noting that $\pi_m(\Gamma_b^{\mathbb{N}})=D(m^p\pi_m(\Gamma_b),m^p)$, since this is the set of $x$ that can be expanded in base $m$ as a concatenation of blocks from $\Gamma_b$, which is the same as restricting the digits in base $m^p$ to the digits $m^p\cdot \pi_mv$, $v\in\Gamma_b$, which is  $\Lambda$. This proves the first part; the second is immediate from the fact that $\pi_n(j,j,j,\ldots)$ has a unique expansion (even if $j=0,n-1$) and period $p=1$.
\end{proof}

\begin{Lemma}\label{lem:size-of-intervals-on-slices}
  Let $\Lambda =\{j\in [n]\,:\,\Gamma_j=[m]\}$ and suppose that $|\Lambda|<|P_2\Gamma|$. If $\eta\in\Omega_n$ and $\eta_k\notin\Lambda$ then any sub-interval  of $\pi_n(\widetilde{F}_\eta)$ is of length at most $1/m^{k}$. In particular, \begin{enumerate}
  \item If $\eta_k\notin\Lambda$ for infinitely many $k$ then $\pi_n(\widetilde{F}_\eta)$ has empty interior.
  \item If $y=\pi_n(\eta)$ with  $\eta_k\notin\Lambda$, and if either $y$ has a unique expansion, or if the other expansion $\eta'$ terminates in a symbol that is not in $\Lambda$, then any sub-interval of $F_y$ has length at most $1/m^k$.
\end{enumerate}
 \end{Lemma}
\begin{proof}
  By assumption $\eta_k\notin\Lambda$ so $\Gamma_{\eta_k}\neq [m]$, and we can choose $i\in [m]\setminus \Gamma_{\eta_k}$. The set of numbers $y$ such that every base-$m$ expansion of $y$ has $i$ in the $k$-th digit is a collection of open intervals whose endpoints are the $m$-adic numbers $(pm+i)/m^k$, $p\in\mathbb{Z}$ ; the union of these intervals is $1/m^k$-dense in $[0,1]$, and contained in the complement of $\pi_m(\widetilde{F}_\eta)$. This proves the main part of the lemma, and conclusion (1) is immediate. Conclusion (2) follows from the fact that if $y$ has a unique expansion $\eta$ then $F_y= \pi_m ( \widetilde{F}_\eta)$; whereas in the case it has two expansions, the hypothesis and (1) guarantee that $\pi_n(\widetilde{F}_{\overline{\eta}})$ has empty interior, and so the length of sub-intervals of  $F_y=\pi_n(\widetilde{F}_\eta)\cup\pi_n(\widetilde{F}_{\overline{\eta}})$ are determined by the first set in the union.
\end{proof}

  \subsection{Three remaining cases}
So far we have reduced Theorem \ref{Main Theorem} to proving that if $g(x)=\alpha x+t$ and $g(F)\subseteq F$ then $\alpha = 1$. We conclude the proof by reducing our problem to  a one-dimensional problem about deleted digit sets. The first ingredient in this strategy is the observation that since $g$ is a homothety, it induces an action on the projections and slices of $F$. Specifically, writing as usual $t=(t_1,t_2)$, let \begin{eqnarray*}
  g_1(x) &=& \alpha x + t_1\\
  g_2(y) &=& \alpha y+t_2
\end{eqnarray*}
Then for $i=1,2$ we have $g_iP_i=P_ig$, which implies that $g_i$ maps the projection $P_iF$ into itself. Moreover, it is easy to check that the vertical line $\{x\}\times \mathbb{R}$ is mapped by $g$ into the vertical line $\{g_1(x)\}\times\mathbb{R}$, and it acts on this line as $g_2$; similarly, the  horizontal line $\mathbb{R}\times\{y\}$ is mapped by $g$ to $\mathbb{R}\times\{g_2(y)\}$, and $g$ acts on this line as $g_1$. Using $gF\subseteq F$, it follows that \begin{eqnarray}
  g_2(F^x) &\subseteq& F^{g_1(x)} \label{eq:x-slice-action}\\
  g_1(F_y) &\subseteq& F_{g_2(y)} \label{eq:y-slice-action}
\end{eqnarray}

The second ingredient is that the projections of $F$ to the axes are deleted digit sets (elementary), and so are some of its horizontal and vertical slices (see Lemma \ref{lem:self-similar-slices}). We may hope to apply the logarithmic Commensurability theorem (Theorem \ref{LCT}) or a similar result to them, to obtain algebraic information about $\alpha$.

This is best demonstrated by example.  Consider the case when $0<\dim P_iF<1$ for $i=1,2$. Note that the lower bound is automatic from our assumption that $F$ is not a product set. Then since $g_1$ maps $P_1F=D(P_1\Gamma,m)$ into itself, Theorem \ref{LCT} implies that $\alpha$ is a rational power of $m$. The same argument applied to $g_2$ and $P_2F$ shows that $\alpha$ is a rational power of $n$. But since $\log m/\log n\notin\QQ$, this is only possible if $\alpha=1$.

To deal with the general case when one or both projections of $F$ are intervals, our analysis now splits into three cases which are based on the nature of the  horizontal fibers:
\begin{enumerate}
\item[A.]  $|\Gamma_j|<m$ for all $j\in [n]$ (equivalently, no horizontal slice contains intervals).
\item[B.]  There is a unique $j$ with $|\Gamma_j|=m$ (equivalently, there is a countable infinity of horizontal slices containing intervals).
\item[C.] There are at least two $j$'s with $|\Gamma_j|=m$ (equivalently, uncountably many horizontal slices contain intervals).
\end{enumerate}

It is clear that this covers all possibilities. The equivalence with the statements in parentheses about the dimension and topology of slices follows from Lemma \ref{lem:size-of-intervals-on-slices}, since $F_y$ has non-empty interior if and only if $y$ has a base-$n$ expansion $\xi$ such that  $\Gamma_{\xi_k}=[m]$ for all large enough $k$, and the set of such $\xi$ will be either empty (if $|\Gamma_j|<m$ for all $j$), countable (if $\Gamma_j=[m]$ for a unique $j$), or uncountable (if $\Gamma_j=[m]$ for at least two values of $j$).

The proof of cases (A) and (C) are similar to each other, combining analysis of the action of $g$ on projections with its action on slices of $F$. The case (B) is slightly different, and for it we use a more combinatorial analysis of the lengths of intervals contained in slices and how they are transformed under $g$. The proofs are given in the next few section.

\subsection{Proof of case (A)}
We assume no horizontal slice has dimension one. Suppose towards a contradiction that $\alpha<1$. By Proposition \ref{Proposition 3}, this implies $\dim P_1 (F) <1$. Set $K=P_1F$ and $\Lambda=P_1\Gamma$, so that $K=D(\Lambda,m)$. Since $F$ is not a product set, $|\Lambda|\geq 2$, and since $\dim P_1F<1$, we actually have $2\leq |\Lambda|<m$. As we observed in the previous section, $g_1(x)=\alpha x+t_1$ maps $P_1F$ into itself, i.e.
\begin{equation*}
\alpha K + t_1 \subseteq K.
\end{equation*}
By Theorem \ref{LCT}, $\alpha = m^q, q\in \mathbb{Q}$.  By replacing $g$ by $g^p$ we can replace $\alpha$ by $\alpha^p$, and therefore we can assume that $\alpha = 1/m^l$ for some $l\in \mathbb{N}$. By Lemma \ref{Lemma - translation of self embedding} we know that $t_1$ is $n$-adic rational. The proof is completed by the following proposition, which we state separately for re-use.

\begin{Proposition}\label{Prop:common-part-of-first-prop}
If $g(F)\subseteq F$ with $g(x)=x/m^l+t$ for some $l\in\NN$, and if $t_1$ is $m$-adic rational, then $g$ is an isometry.
\end{Proposition}
\begin{proof}

By Proposition \ref{Proposition 3}, some vertical slice of $F$ has positive dimension, so there exists $j_1\in[m]$ with $1<|\Gamma^{j_1}| \leq n$. To simplify the presentation, assume that $j_1 \neq 0,m-1$ and that  $1<|\Gamma^{j_1}| < n$. We explain how to eliminate these assumptions at the end of the proof.

Write \[
  \Lambda' = \Gamma^{j_1} \subseteq [n]\;\;\;\;\textrm{and}\;\;\;\;K'=D(\Lambda',n)
\]
Note that $0< \dim K <1$ since $1<|\Lambda'|<n$. Also define
\begin{equation*}
\xi=(j_1, j_1, j_1, j_1, j_1,...)\in\Omega_m\;\;\;\;\textrm{and}\;\;\;\; x=\pi_m(\xi)=\sum_{k=1} ^\infty \frac{j_1}{m^k}
\end{equation*} 
Then $x\in P_1F$ and, since $\xi$ is the unique expansion of $x$, by Lemma \ref{lem:self-similar-slices}, $F^x=D(\Lambda',n)$.

Write $x'=g_1(x)=x/m^l+t_1$. By \eqref{eq:x-slice-action} we have \begin{equation}
  \frac{1}{m^l} F^x + t_1 \subseteq F^{x'}. \label{eq:vertical-slice-transformation}
\end{equation}
By assumption $t_1$ is $n$-adic rational, whereas $x$ (and hence $x/m^l$) is not (using our assumption $j_1\neq 0,n-1$). Therefore $x'$ is not $n$-adic rational, and has a unique expansion $\xi'\in(\Lambda')^\NN$ terminating in $j_1$'s; in fact we will have \[
  \xi'_k = j_1\;\;\;\;\;\;\textrm{for all }k\geq l+2
  \]
  Therefore by Lemma \ref{lem:self-similar-slices} again, $F^{x'}$ is a union of finitely many translates of $\frac{1}{n^{l+1}}K$, so \eqref{eq:vertical-slice-transformation} and Corollary \ref{Corollary from LCT}  imply that $1/m^l$ is a rational power of $n$. Since $\log m/\log n\notin\mathbb{Q}$, this is impossible.

We now explain how to treat the case when $j_1=0$ or $m-1$, or that  $|\Lambda| =n$. Let $j_2 \neq j_1$ be such that $1\leq |\Gamma^{j_2}|<n$. Such a $j_2$ exists since otherwise either $F$ sits on a vertical line, or $F$ is a product set $P_1 (F) \times [0,1]$. We then define 
\begin{equation*}
\xi=(j_2, j_1, j_2, j_1, j_2, j_1, j_2,...)\in\Omega_m \;\;\;\;\textrm{and}\;\;\;\;x=\pi_m(\xi)
\end{equation*}
Then using Lemma \ref{lem:self-similar-slices},  $F^x$ is the deleted digit set $K'=D(n\Gamma^{j_1}+\Gamma^{j_2},n^2)$, and defining $\xi'$ as before, we similarly find that $F^{x'}$ is a finite union of copies of $K'$ scaled by $n^{l+1}$. We use here the fact that numbers with base-$m$ expansion ending in alternating $j_1$ and $j_2$'s have a unique expansion. The rest of the argument is the same.

\end{proof}

The last proposition applies equally well with the roles of $m,n$ reversed. Indeed, the proof did not use the fact that $m>n$, and while it did rely on Proposition \ref{Proposition 3}, whose conclusion is asymmetric in $m,n$, we only used the part of its conclusion which gives a slice of positive (rather than full) dimension, which applies to both $m$ and $n$.

\subsection{Proof of case (C)}
We now assume that
\begin{equation*}
\Lambda = \lbrace j\in [n] \;|\; \Gamma _j = [m] \rbrace 
\end{equation*}
satisfies $|\Lambda| \geq 2$. We also have $|\Lambda|<n$ since otherwise $F=[0,1]\times P_2F$, which would put us in the product case. Assume towards a contradiction that $\alpha<1$.

Let $K = D(\Lambda  , n)$, so by the bounds on $|\Lambda|$ we have $0<\dim K<1$. For any $z\in K$ the horizontal slice $F_z$ is an interval of length $1$, hence $g(F_z\times\{z\})$ is an interval of length $\alpha$ and, as we saw earlier, it is contained in the slice $F_{z'}$ for $z'=\alpha z+t_2$.  Choosing $p\in\mathbb{N}$ such that
\begin{equation} \label{Equation for p -new1}
\frac{1}{m^p} \leq \alpha < \frac{1}{m^{p-1}}
\end{equation}
we have found that $F_{z'}$ contains an interval of length at least $1/m^p$. By Lemma \ref{lem:size-of-intervals-on-slices}, this means that if $z'$ has a unique base-$n$ expansion then it belongs to the set\[
  E=\lbrace \sum_{k=1}^\infty\frac{y_k}{n^k}\;:\;y_k\in\Lambda\textrm{ for }k\geq p+1\rbrace
\]  
which is a finite union of translates of sets $\frac{1}{n^p}K$. Let $K_0\subseteq K$ denote those $z$ for which $z'$ has  multiple expansions; we have thus shown that  $z\mapsto z'$ maps $K\setminus K_0$ into $E$. But since $K$ is perfect we have $K=\overline{K\setminus K_0}$, and continuity of $z\mapsto z'$ and the fact that $E$ is closed imply that   \begin{equation} \label{equation for K1}
\alpha K + t_2 \subseteq E
\end{equation}
It follows from  Corollary \ref{Corollary from LCT} that $\alpha = n^q$ for some $q \in \mathbb{Q}$. Replacing $g$ by a power of $g$  if necessary, we may assume that $\alpha = \frac{1}{n^l}$ for some $l\in \mathbb{N}$ (we keep all other notation the same, but now $p$ and $E$ are chosen with respect to this power of $g$). By Corollary \ref{Cor:generalized-embedding-of-dd-sets}, $t_2$ is $n$-adic rational, and we can apply Proposition \ref{Prop:common-part-of-first-prop} (with the roles of $m,n$ reversed; see remark after the proof of the Proposition) and conclude that $\alpha=1$.

\subsection{Proof of case (B)}\label{subsec:case-B}

The remaining case is when there exists a unique index $j_1\in[n]$ such that $\Gamma_{j_1}=[m]$. 
We can assume that $0\in P_2F$ (equivalently, $0\in P_2\Gamma$), since if not we can apply a translation to make it so. We can also assume that $j_1\neq 0$, since there certainly are other indices in $P_2\Gamma$ (otherwise $F$ would be contained in a horizontal line), and if $j_1=0$  we could apply a reflection through the $x$-axis (thus reversing the order of vertical slices),  followed by a translation (to ensure again $0\in P_2F$). This modification replaces $g$ with the composition of a homothety and a reflection, but we can pass to $g^2$ again and work with it. We assume that these adjustments have been made.

Suppose towards a contradiction that $\alpha <1$. Let $p\in \mathbb{N}$ be the unique integer such that
\begin{equation} \label{alpha1}
\frac{1}{m^p} \leq \alpha \leq \frac{1}{m^{p-1}}.
\end{equation}
We can assume that
\begin{equation} \label{p is large1}
\frac{m^{p-1}}{n^{p+1}} >1,
\end{equation}
since this holds whenever $p$ is large enough (equivalently, $\alpha$ small enough), so if it was not initially the case, we can just replace $g$ by some power $g^k$ (and consequently  $\alpha$ by $\alpha^k$). 

Let \[
z=\pi_n(j_1,j_1,j_1,\ldots)
\]
This means that \[
  F_z=[0,1]
  \]
 As already noted, $g$ maps horizontal slices into horizontal slices, and in particular, writing $y=\alpha z+t_2$ we have \[
  \alpha F_z + t_1 \subseteq F_y
  \]
  Thus $F_y$ contains an interval of length $\alpha\geq 1/m^p$. Since  $j_1$ is the only digit with $\Gamma_j=[m]$, it follows from Lemma \ref{lem:size-of-intervals-on-slices} (2) that \[
  y= \pi_n(y_1,y_2,\ldots y_p,j_1,j_1,j_1,\ldots)
  \]
for some choice of $y_k \in P_2\Gamma$.

Next let  \[
z'= \pi_n(0,j_1,j_1,j_1,\ldots)
\]
which, since $0\in P_2\Gamma$, is an element of $P_2F$. Arguing as above we find that $F_{z'}$ contains an interval of length  $\frac{1}{m}$,  and, setting  $y'=\alpha z' +t_2$, we see that $F_{y'}$ contains an interval of  length $\alpha / m\geq 1/m^{p+1}$, and so \[
  y'=\pi_n(y'_1,y'_2,\ldots,y'_{p+1},j_1,j_1,j_1,\ldots)
  \]
for some $y_k ' \in [n]$,

Now, on the one hand, since $z,z'$ differ only in their first coordinates, which are respectively $j'$ and $0$, and since $\alpha\leq 1/m^{p+1}$ and $j_1\leq n$,\[
|y-y'|=\alpha|z-z'|=\alpha \frac{j_1}{n}\leq \frac{1}{m^{p-1}}
\]
On the other hand, the expansions of $y,y'$ agree from in their $p+2$-th digit and differ in their first $p+1$-th digits, since they are not equal, being images of distinct points $z,z'$, so \[
|y-y'|\geq \frac{1}{n^{p+1}}
\]
The last two inequalities contradict \eqref{p is large1}. This completes the proof.

\section{Proof of Theorem \ref{Theorem - structure of tangnet sets}} \label{Section - proof of Theorem tangent sets}
\subsection{Preliminaries and notation} \label{New notation}

Let $F\subseteq[0,1]^{2}$ be a Bedford-McMullen carpet defined by integers
$m>n\geq2$ and digit set $\Gamma\subseteq[m]\times[n]$.

To avoid making frequent exceptions for the empty set, it is convenient
to define the Hausdorff distance between the empty and every other
 compact subset of $Q=[-1,1]^2$ to be $1+\text{diam}(Q)$. With
this definition, we have added a single isolated point to the space $\cmpct(Q)$
of compact subsets of $Q$, preserving compactness.

Given $l\in\mathbb{N}$ and $b\in[n]^{l}$, let 
\begin{eqnarray*}
H(b) & = & \bigcup_{a\in\Gamma_{b}}\left(\left(\begin{array}{cc}
m^{-l} & 0\\
0 & 1
\end{array}\right)F+\left(\begin{array}{c}
\pi_{m}(a)\\
0
\end{array}\right)\right)
\end{eqnarray*}
This is a union of copies of $F$ which have been contracted by $m^{l}$
in the horizontal direction. Each set in the union is contained in
a rectangle of height $1$ and width $1/m^{l}$ with lower left
corner at the $m$-adic rational vector $(\pi_{m}(a),0)$ for some
$a\in\Gamma_{b}$. 

Looking ahead to our identification of $m$-adic tangent sets, we note
that the map $\eta\mapsto H(\eta_{1}\ldots\eta_{l})$
is continuous from $\Omega_{n}$ to $\cmpct(Q)$.
Also,
\begin{Lemma}
For $\eta\in\Omega_{n}$ we have $H(\eta_{1}\ldots\eta_{l})\rightarrow\pi_{m}(\widetilde{F}_{\eta})\times P_{2}(F)$
in the Hausdorff topology.
\end{Lemma}
\begin{proof}
  Follows by a straightforward calculation, using item Proposition \ref{Proposition - Hausdorff metric0}\eqref{H-limit-pointwise}. 
\end{proof}

Next, given $b\in[n]^{l}$ and $1\leq s\leq 1$, we say that $E\subseteq[-2,2]^2$
is a basic $(b,s)$-set if it is of the form 
\[
\left(\begin{array}{cc}
1\\
 & n^s
\end{array}\right)\cdot\left(H(b)+z\right)
\]
for some $z\in\mathbb{R}^2$. For $\eta\in\Omega_{n}$ we have already defined
an $(\eta,s)$-set to be a set of the form
\[
\left(\begin{array}{cc}
1\\
 & n^s
\end{array}\right)\cdot\left(\pi_{m}(\widetilde{F}_{\eta})\times P_{2}(F)+z\right)
\]
which is contained in $[-2,2]^{2}$.
The last lemma and Proposition \ref{Proposition - Hausdorff metric0}\eqref{H-limit-on-open-set} now give:

\begin{Lemma}
\label{lem:HAusdorff-convergence-on-Q}Let $\eta\in\Omega_{n}$, let
$E_{i}$ be $(\eta_{1}\ldots\eta_{l_{i}},s_{i})$-sets with $l_{i}\rightarrow\infty$
and $0\leq s_{i}\leq 1$. If $E_{i}\rightarrow E$ in the Hausdorff
topology, then $E$ is an $(\eta,s)$-set, and if $E_{i}\cap[-1,1]^{2}\rightarrow E'$,
then $E\cap(-1,1)\subseteq E'\subseteq E$.
\end{Lemma}

\subsection{Structure of $m$-adic mini-sets}

For $a,b \in [m]^l \times [n]^k$ where $k,l\geq 0$ are integers, we use the notation, 
\begin{equation*}
[a]\times [b] = \lbrace (\omega, \eta) \in \Gamma^\NN : (\omega_1\ldots\omega_l , \eta_1\ldots\eta_k) = (a,b) \rbrace.
\end{equation*}
Notice that this is not the product of cylinders in $([m]\times[n])^\NN$, but rather in $\Gamma^\NN$, so the projection $\pi_{m,n}([a]\times[b])$ is an affine image of $F$, not of the unit square (for a more precise statement see the proof of Lemma \ref{lem:mini-sets} below). Also, there is notational conflict with the notation $[m]=\{0,\ldots,m-1\}$, but which is meant should be clear from the context).

Fix $f=(f_{1},f_{2})\in F$ and $l\in\mathbb{N}$. We consider
the $m$-adic mini-set 
\[
T(F,f,m,l)= [ m^l ( F -f)]\cap Q =m^{l}((F-f)\cap m^{-l}Q)
\]
Define 
\[
k=\left\lfloor l\log_{n}m\right\rfloor 
\]
(we suppress the dependence on $l$ in the notation). This is the
unique integer satisfying 
\[
m^{-l}<n^{-k}\leq m^{-(l-1)}
\]
Since $m>n$ we have $k\geq l$, and $k>l$ for large enough
$l$. 

In what follows, for $a\in[m]^{k}$ we write $a=a'a''$ with $a'\in[m]^{l}$
and $a''\in[m]^{k-l}$. We similarly write $b=b'b''$ when $b\in[n]^{k}$.
\begin{Lemma}
Let $f\in F$. If $a'\in[m]^{l}$ and $b\in[n]^{k}$ and if $\pi_{m,n}([a']\times[b])$
intersects $f+m^{-l}Q$ non-trivially, then $|\pi_{m}a'-f_{1}|\leq2m^{-l}$
and $|\pi_{n}b-f_{2}|\leq2m^{-l}$, and given $f$, there are at most four
possibilities for the prefix $a'$ and at most four possibilities
for $b$.
\end{Lemma}
\begin{proof}
Elementary using $k\geq l$ and the fact an interval of length $c\cdot m^{-l}$
contains at most $\left\lfloor c\right\rfloor +4$ reduced rationals
with denominator $m^{l}$ and at most $\left\lfloor c\right\rfloor +4$
reduced rationals with denominator $n^{k}$. 
\end{proof}

By the lemma, we can partition the cylinder sets $\pi_{m,n}([a]\times[b])$
which intersect $f+m^{-l}Q$ into finitely many classes according
to the sequence $b$ and the prefix $a'=a_{1}\ldots a_{l}$ of
$a$. Specifically, write $C_{a,b}=\pi_{m,n}([a]\times[b])$ and for
$b\in[n]^{k}$ and $a'\in[m]^{l}$ let 
\[
\mathcal{C}(f,a',b)=\left\{ C_{a,b}\,:\,a'\mbox{ is a prefix of }a\in[m]^{k}\mbox{ and }C_{a,b}\cap(f+m^{-l}Q)\neq\emptyset\right\} 
\]
Note that $\mathcal{C}(f,a',b)\neq\emptyset$ precisely when there exists $a''\in[m]^{k-l}$
with $C_{a'a'',b}\cap(f+m^{-l}Q)\neq\emptyset$, and the latter happens if and only if $(C_{a'a'',b}-f)\cap(m^{-l}Q)\neq\emptyset$.
\begin{Lemma}\label{lem:mini-sets}
Let $l,k$ be as above. Let $b=b'b''\in[n]^{k}$ and $a'\in[m]^{l}$.
Let $\mathcal{C}=\mathcal{C}(f,a',b)$ and $r=\log_n m^{l}/n^{k}=l\log_nm\mod 1$. If
$\mathcal{C}\neq\emptyset$ then $m^{l}(\cup\mathcal{C}-f)$ is a
basic $(b'',s)$-set.
\end{Lemma}
Note that by choice of $k$ we have $1<r\leq m$. 
\begin{proof}
Given $a''\in\Gamma_{b''}$ set $a=a'a''\in[m]^{k}$, and write $\alpha=\pi_{m}(a)$
$\alpha'=\pi_{m}(a')$ and $\alpha''=\pi_{m}(a'')$. Similarly define
$\beta=\pi_{n}(b)$, $\beta'=\pi_{n}(b')$ and $\beta''=\pi_{n}(b'')$.
In our discussion $b',b''$ and $a'$ are fixed and hence $\alpha',\beta,\beta',\beta''$
are well defined, but $a''$ will vary, and so $\alpha,\alpha''$
are implicitly functions of $a''$, a fact which we suppress in the
notation. Now, 
\begin{eqnarray*}
\cup\mathcal{C} & = & \bigcup_{a''\in\Gamma_{b''}}\left(\pi_{m,n}([a'a'']\times[b'b''])\right)\\
 & = & \bigcup_{a''\in\Gamma_{b''}}\left(\left(\begin{array}{cc}
m^{-k} & 0\\
0 & n^{-k}
\end{array}\right)F+(\pi_{m}(a'a''),\pi_{n}(b'b''))\right)\\
 & = & \bigcup_{a''\in\Gamma_{b''}}\left(\left(\begin{array}{cc}
m^{-k} & 0\\
0 & n^{-k}
\end{array}\right)F+(\alpha,\beta)\right)
\end{eqnarray*}
Using the fact that $\alpha=\alpha'+m^{-l}\alpha''$ and $\beta=\beta'+n^{-l}\beta''$,
and writing $z=m^{l}(\alpha'-f_1,\beta-f_2)$, we get 
\begin{eqnarray*}
m^{l}(\cup\mathcal{C}-f) & = & \bigcup_{a''\in\Gamma_{b''}}\left(m^{l}\left(\begin{array}{cc}
m^{-k}\\
 & n^{-k}
\end{array}\right)F+m^{l}(\alpha,\beta) - m^{l}(f_1,f_2) \right)\\
 & = & \bigcup_{a''\in\Gamma_{b''}}\left(\left(\begin{array}{cc}
m^{-(k-l)} & 0\\
0 & m^{l}/n^{k}
\end{array}\right)F+m^{l}(\alpha'-f_1,\beta-f_2)+(\alpha'',0)\right)\\
 & = & \bigcup_{a''\in\Gamma_{b''}}\left(\left(\begin{array}{cc}
1 & 0\\
0 & n^s
\end{array}\right)\left(\begin{array}{cc}
m^{-(k-l)} & 0\\
0 & 1
\end{array}\right)F+z+(\alpha'',0)\right)\\
 & = & \bigcup_{a''\in\Gamma_{b''}}\left(\left(\begin{array}{cc}
1 & 0\\
0 & n^s
\end{array}\right)\left(\left(\begin{array}{cc}
m^{-(k-l)} & 0\\
0 & 1
\end{array}\right)F+(\alpha'',0)\right) + z\right)\\
 & = & \left(\begin{array}{cc}
1 & 0\\
0 & n^s
\end{array}\right)(H(b)+z)
\end{eqnarray*}
Recalling that $\alpha''=\pi_{m}(a'')$, this shows that $m^{l}(\cup\mathcal{C}-f)$
is a basic $(b'',s)$-set. 
\end{proof}

Given $f,m,l$ and $k$ as above, set
\[
B(f,l)=\left\{ b''\in[n]^{k-l}\,:\,\begin{array}{c}
\exists a\in[m]^{l}\,,\,b'\in[n]^{l}\,\mbox{such that}\\
\pi_{m,n}([a]\times[b'b''])\cap(f+m^{-l}Q)\neq\emptyset
\end{array}\right\} 
\]
Combining the last two lemmas,   we have:
\begin{Corollary} \label{Corollary X}
$|B(f,l)|$ is uniformly bounded in $f,l$  and $T(F,f,m,l)$
is the intersection of $Q$ with the union over $b''\in B(f,l)$
of basic $(b'',s)$-sets, with $s=l\log_nm\bmod 1$
\end{Corollary}

\subsection{Proof of Theorem \ref{Theorem - structure of tangnet sets}}

Recall that for a given $l\in\mathbb{N}$ we defined $k=k(l)=\left\lfloor l\log_{n}m\right\rfloor $
as before, $r(l)=m^{l}/n^{k}$. Set $s(l)=\log_{n}r(l)=l\log_{n}m\bmod1$,
and $(s(l))_{l=1}^{\infty}$. Note that $s(l)$  is just the orbit of $0$ under the rotation
of $\mathbb{R}/\mathbb{Z}$ by $\log_{n}m$.

Recall the definition of $S(u),S'(u)$ and of $(\eta,s)$-multisets from Section \ref{subsec:tangent-sets}. We recall for convenience the theorem we are out to prove:

\begin{theorem*}
Fix $f=(f_{1},f_{2})\in F$ with $f_2\neq 0,1$ and let $\eta\in\pi_n^{-1}(f_2)$. Then for every $m$-adic tangent set
$T\in T(F,f,m)$, there exists $(\xi,s)\in S(\eta)$ such that $T$ is a non-empty union of a $(\xi,s)$-multiset and a $(\overline{\xi},s)$-multiset. Conversely, if $(\xi,s)\in S(f_{2})$, then there is an
$m$-adic tangent set $T\in T(F,f,m)$ which a union of this type.

In the special case when $f_2=0$ or $f_2=1$, the same is true but omitting the $(\overline{\xi},s)$-multiset from the union.
\end{theorem*}
\begin{proof}
We prove the first statement, though it applies also in the case $f_2=0,1$  (only the conclusion is slightly weaker). Note that the property that $T$ is a union of a $(\xi,s)$-multiset and $(\overline{\xi},s)$-multiset is the same as saying that for some $y$ it is a union of finitely many $(\xi,s)$-sets with $\pi_n(\xi)=y\bmod 1$. We shall prove this version.

  Fix $T\in T(F,f,m)$ and let $l(i)\rightarrow\infty$ be such that
$E_{i}=m^{l(i)}(F-f)\cap Q\rightarrow T$ , as $i\rightarrow\infty$,
in the Hausdorff metric. Write $s(i)=l(i)\cdot\log_{n}m\bmod1$,
so by Corollary \ref{Corollary X}, $T(F,f,m,l(i))$ is the union of basic
$(b'',{s(i)})$-sets for $b''\in B(f,l(i))\subseteq[n]^{k(i)-l(i)}$. 

Passing to a subsequence if necessary, we may assume that  $n^{l(i)}f_{2}\rightarrow y$
in $\mathbb{R}/\mathbb{Z}$ and that $s(i)\rightarrow s$ in $\mathbb{R}/\mathbb{Z}$.
Since the sets $B(f,l(i))$ are finite and of bounded cardinality,
by possibly passing to a further sub-sequence, it follows from Lemma
\ref{lem:HAusdorff-convergence-on-Q} and Proposition \ref{Proposition - Hausdorff metric0} item 2 that there are finitely many
$(\eta,s)$-sets $E_{j}$, for certain $\eta\in\Omega_{n}$, such
that \eqref{eq:2} holds, and the left hand side is not empty because it contains $0$. It remains only to show that $\eta\in\pi_{n}^{-1}(y)$.

Indeed, consider one of the basic $(\eta,s)$-set $E_{j}$ in the
union \eqref{eq:2} that $T$ satisfies. This means that there exist $b(i)=b'(i)b''(i)\in[n]^{k(i)}$
(so $b''(i)\in[n]^{k(i)-l(i)}$) such that $b''(i)\rightarrow\eta$
in the obvious sense, such that 
\begin{equation}
|\pi_{n}(b(i))-f_{2}|\leq2m^{-l(i)}\label{eq:1}
\end{equation}
and such that $E_{j}$ is the limit of basic $(b''(i),s(i)$-sets.
Since the shift $\sigma_n$ on $\Omega_{n}$ is semi-conjugated by $\pi_{n}$
to multiplication by $n$ on $\mathbb{R}/\mathbb{Z}$ and $b''(i)=\sigma^{l(i)}(b(i))$,
\eqref{eq:1} implies that as $i\rightarrow\infty$, 
\[
|\pi_{n}(b''(i))-n^{l(i)}f_{2}|=|n^{l(i)}\pi_{n}(b''(i))-n^{l(i)}f_{2}|\leq2\frac{n^{l(i)}}{m^{l(i)}}
\]
where $\left|\cdot\right|$ denotes distance in $\mathbb{R}/\mathbb{Z}$.
The right hand side is $o(1)$ as $i\rightarrow\infty$. Since $b''(i)\rightarrow\eta$
and $n^{l(i)}b(i)\rightarrow y\bmod1$, and since $\pi_{n}$ is
continuous, this means that $\pi_{n}(\eta)=y$, as claimed.

The converse direction is similar: starting from the sequence $l(i)$
such that ,
\begin{equation*}
(n^{l(i)}f_{2},l(i)\log_{n}m)\rightarrow(y,s)\bmod1
\end{equation*}
we pass to a further subsequence so that $T(F,f,m,l(i))\rightarrow T \in T(F,f,m)$
for some $T$. The same analysis then shows that $T$ is a non-empty union of a $(\xi,s)$-multiset and $(\overline{\xi},s)$-multiset for $\xi\in\pi_n. ^{-1}(y)$.

Finally, in the case $f_2 =0,1$, the same analysis applies. The only difference is that \eqref{eq:1} implies that $b''$ agrees on a growing number of digits with the (unique) expansion of $f_2$. From here the argument is the same, but without taking $\bmod 1$.
\end{proof}

\bibliography{bib}{}
\bibliographystyle{plain}

\bigskip
\footnotesize
\begin{description}
  \item[Address] Einstein Institute of Mathematics, Edmond J. Safra Campus, The Hebrew University of Jerusalem, Givat Ram. Jerusalem, 9190401, Israel.
  \item[Email] \texttt{amir.algom@mail.huji.ac.il}
    
    \hspace*{2pt} \texttt{mhochman@math.huji.ac.il}
\end{description}

\end{document}